\documentclass[a4paper, twoside, 10pt, final]{amsart}
\usepackage[top=2.5cm, bottom=2cm, inner=2.5cm, outer=2.5cm, foot=1cm]{geometry}
\usepackage{algorithm,algorithmic,amssymb,amsfonts,longtable,microtype}
\usepackage[utf8]{inputenc}
\usepackage{lmodern}
\usepackage{svninfo}
\usepackage[bookmarks=true,colorlinks=true,linkcolor=blue,pdfa=true]{hyperref}
\usepackage[notcite,notref]{showkeys}

\hypersetup{pdftitle={A characterization of arithmetical invariants by the monoid of relations II: The monotone catenary degree and applications to semigroup rings},pdfauthor={Andreas Philipp}}
\usepackage{aliascnt}
\let\orgautoref\autoref
\renewcommand{\autoref}
        {%
\def\corollaryautorefname{Corollary}%
\def\definitionautorefname{Definition}%
\def\lemmaautorefname{Lemma}%
\def\propositionautorefname{Proposition}%
\def\exampleautorefname{Example}%
\def\remarkautorefname{Remark}%
\def\chapterautorefname{Chapter}%
\def\sectionautorefname{Section}%
\def\subsectionautorefname{Subsection}%
\def\subsubsectionautorefname{Step}%
\def\algorithmautorefname{Algorithm}%
         \orgautoref}

\newtheorem{theorem}{Theorem}[section]
\newaliascnt{lemma}{theorem}
\newtheorem{lemma}[lemma]{Lemma}
\aliascntresetthe{lemma}
\newaliascnt{corollary}{theorem}

\aliascntresetthe{corollary}
\newaliascnt{proposition}{theorem}
\newtheorem{proposition}[proposition]{Proposition}
\aliascntresetthe{proposition}
\theoremstyle{definition}
\newaliascnt{definition}{theorem}
\newtheorem{definition}[definition]{Definition}
\aliascntresetthe{definition}
\newaliascnt{remark}{theorem}

\aliascntresetthe{remark}
\newaliascnt{notation}{theorem}

\aliascntresetthe{notation}
\newaliascnt{situation}{theorem}

\aliascntresetthe{situation}
\newaliascnt{example}{theorem}
\newtheorem{example}[example]{Example}
\aliascntresetthe{example}

\newcommand{\F}{\mathbb F}
\newcommand{\N}{\mathbb N}

\newcommand{\Z}{\mathbb Z}
\newcommand{\Q}{\mathbb Q}

\DeclareMathOperator{\spec}{spec}
\DeclareMathOperator{\supp}{supp}
\DeclareMathOperator{\ord}{ord}
\DeclareMathOperator{\pic}{Pic}

\newcommand{\cat}{\mathsf c}
\newcommand{\cadj}{\cat_{\mathrm{ad}}}
\newcommand{\ceq}{\cat_{\mathrm{eq}}}
\newcommand{\cmon}{\cat_{\mathrm{mon}}}
\newcommand{\mueq}{\mu_{\mathrm{eq}}}
\newcommand{\muadj}{\mu_{\mathrm{ad}}}
\newcommand{\Req}{\mathcal R_{\mathrm{eq}}}
\newcommand{\equ}{\approx_{\mathrm{eq}}}

\newcommand{\punkt}[1]{#1^\bullet}
\newcommand{\mal}[1]{#1^\times}
\newcommand{\mdots}{\cdot\ldots\cdot}

\newcommand{\wmal}[1]{\mal{\widehat{#1}}}

\newcommand{\ldbrack}{[\negthinspace[}
\newcommand{\rdbrack}{]\negthinspace]}

\numberwithin{equation}{section}

\makeatletter
\renewcommand{\p@enumii}{}
\makeatother

\begin{document}
\svnInfo $Id: char-cmon-mon-rel.tex 2234 2011-04-02 06:28:44Z aph $
\svnKeyword $HeadURL: svn+ssh://aph@aph.homelinux.net/var/svn/notes/publications/submitted/char-cmon-mon-rel/char-cmon-mon-rel.tex $

\address{Institut f\"ur Mathematik und Wissenschaftliches Rechnen \\
Karl-Franzens-Universit\"at Graz \\
Heinrichstra\ss e 36\\
8010 Graz, Austria} \email{andreas.philipp@uni-graz.at}

\thanks{I thank my Ph.D. thesis advisors Prof. Franz Halter-Koch and Prof. Alfred Geroldinger for all the help, advice, and mathematical discussions during my thesis which led to all results in this article.}

\author{Andreas Philipp}



\begin{abstract}
The investigation and classification of non-unique factorization phenomena has attracted some interest in recent literature. For finitely generated monoids, S.T.~Chapman and P.A.~García-Sánchez, together with several co-authors, derived a method to calculate the catenary and tame degree from the monoid of relations. Then, in \cite{phil10}, the algebraic structure of this approach was investigated and the restriction to finitely generated monoids was removed. We now extend these ideas further to the monotone catenary degree and then apply all these results to the explicit computation of arithmetical invariants of semigroup rings.
\end{abstract}

\title[A characterization of arithmetical invariants by the monoid of relations II]{A characterization of arithmetical invariants by the monoid of relations II: The monotone catenary degree and applications to semigroup rings}

\maketitle

\section{Introduction}
\label{sec:int}
\bigskip

An integral domain and, more generally, a commutative, cancellative monoid is called \emph{atomic} if every non-zero non-unit has a factorization into irreducible elements, and it is called \emph{factorial} if this factorization is unique up to ordering and associates. Non-unique factorization theory is concerned with the description and classification of non-unique factorization phenomena arising in atomic domains. It has its origin in algebraic number theory---the ring of integers of an algebraic number field being atomic but generally not factorial---but in the last decades it became an autonomous theory with many connections to zero-sum theory, commutative ring theory, module theory, additive combinatorics, and representations of monoids. We refer to the monograph \cite{MR2194494} for a recent presentation of the various aspects of the theory.

To describe these phenomena, various invariants have been studied in the literature. Among these, the tame degree, the catenary degree, and---a variant thereof---the monotone catenary degree received some attention in recent research; for some new results, see, e.g. \cite{garger10}, \cite{MR2660910}, and \cite{Ge-Gr-Sc11a}; for an overview of known results and additional references, see, e.g., the monograph \cite{MR2194494}; for a statement of the formal definitions, see \autoref{sec:pre} and the beginning of \autoref{sec:mstep}. Additionally, monotone and near monotone chains of factorizations have been studied in \cite{MR2140688}, \cite{MR2208111}, and \cite[Section 7]{MR2660910}.

For an integral domain, non-unique factorization phenomena only concern the multiplicative monoid of that domain. Thus we will only derive the theory for commutative, cancellative monoids, and apply these results again to domains afterwards.

The monoid of relations associated to a monoid and a certain invariant $\mu(\cdot)$ have been used successfully to study the catenary degree and other invariants. Investigations of this type started only fairly recently. In \cite{MR1719711}, such investigations were carried out for finitely generated monoids using results from \cite{MR2494887} and \cite{MR2254337}. In \cite{MR2243561} and \cite{Om}, these results and expansions thereof were applied in the investigation of numerical monoids; for a detailed exposition of numerical monoids and applications, see, e.g. the monograph \cite{MR2549780}. In \cite{phil10} the algebraic structure of this method was studied: i.e., the invariant $\mu(\cdot)$, its definition, and the monoid of relations. By this more algebraic-structural approach, the results could be extended to not necessarily finitely generated monoids.

In the present paper, we extend the tools from \cite{phil10} to study the monotone catenary degree by submonoids of the monoid of relations. In \autoref{sec:comp}, we apply these results and many of the results from \cite{phil10} to the explicit computation of arithmetical invariants for various semigroup rings. Additionally, the arithmetic of some some generalized power series rings is studied there.

Moreover, these abstract characterizations, and, in particular, \autoref{3.4} are used successfully for investigations on the arithmetic of non-principal orders in algebraic number fields in \cite{phil11b}.

\bigskip
\section{Preliminaries}
\label{sec:pre}
\bigskip

In this note, our notation and terminology will be consistent with \cite{MR2194494}. Let $\N$ denote the set of positive integers and let $\N_0=\N\uplus\lbrace 0\rbrace$. For integers $n,\,m\in\Z$, we set $[n,m]=\lbrace x\in\Z\mid n\leq x\leq m\rbrace$. By convention, the supremum of the empty set is zero and we set $\frac{0}{0}=1$. The term ``monoid'' always means a commutative, cancellative semigroup with unit element. We will write all monoids multiplicatively. For a monoid $H$, we denote by $H^\times$ the set of invertible elements of $H$. We call $H$ reduced if $H^\times=\lbrace 1\rbrace$ and call $H_{\mathrm{red}}=H/H^\times$ the reduced monoid associated with $H$. Of course, $H_{\mathrm{red}}$ is always reduced.
Note that the arithmetic of $H$ is determined by $H_{\mathrm{red}}$ and therefore we can restrict our attention to reduced monoids whenever convenient. We denote by $\mathcal A(H)$ the \emph{set of atoms} of $H$, by $\mathcal A(H_{\mathrm{red}})$ the set of atoms of the associated reduced monoid $H_{\mathrm{red}}$, by $\mathsf Z(H)=\mathcal F(\mathcal A(H_{\mathrm{red}}))$ the free (abelian) monoid with basis $\mathcal A(H_{\mathrm{red}})$, and by $\pi_H:\mathsf Z(H)\rightarrow H_{\mathrm{red}}$ the unique homomorphism such that $\pi_H|\mathcal A(H_{\mathrm{red}})=\mathrm{id}$. We call $\mathsf Z(H)$ the \emph{factorization monoid} and $\pi_H$ the \emph{factorization homomorphism} of $H$. For $a\in H$, we denote by $\mathsf Z(a)=\pi_H^{-1}(a\mal H)$ the \emph{set of factorizations} of $a$ and denote by $\mathsf L(a)=\lbrace |z|\mid z\in\mathsf Z(a)\rbrace$ the \emph{set of lengths} of $a$, where $|\cdot|$ is the ordinary length function in the free monoid $\mathsf Z(H)$. In this terminology, a monoid $H$ is called \emph{half-factorial} if $|\mathsf L(a)|=1$ for all $a\in H\setminus\mal H$---this coincides with the classical definition of being half-factorial, since then every two factorizations of an element have the same length---and \emph{factorial} if $|\mathsf Z(a)|=1$ for all $a\in H\setminus\mal H$.

With all these notions at hand, for $a\in H$, we set
\[
\rho(a)=\frac{\sup\mathsf L(a)}{\min\mathsf L(a)}\:\mbox{ and call }\:
\rho(H)=\sup\lbrace\rho(a)\mid a\in H\rbrace\:\mbox{ the \emph{elasticity} of }H.
\]
Note that $H$ is half-factorial if and only if $\rho(H)=1$.

For two factorizations $z,\,z'\in\mathsf Z(H)$, we call
\[
\mathsf d(z,z')=\max\left\lbrace\left|\frac{z}{\gcd(z,z')}\right|,\left|\frac{z'}{\gcd(z,z')}\right|\right\rbrace
\quad\mbox{the \emph{distance} between }z\mbox{ and }z'
\]
and, for two subset $X,\,Y\subset\mathsf Z(H)$, we call
\[
\mathsf d(X,Y)=\min\lbrace\mathsf d(x,y)\mid x\in X,\,y\in Y\rbrace
\quad\mbox{the \emph{distance} between }X\mbox{ and }Y.
\]
If one of the sets is a singleton, say $X=\lbrace x\rbrace$, we write $\mathsf d(\lbrace x\rbrace,Y)=\mathsf d(x,Y)$.\\
Let $a\in H$. We call two lengths $k,\,l\in\mathsf L(a)$ with $k<l$ \emph{adjacent} if $[k,l]\cap\mathsf L(a)=\lbrace k,l\rbrace$ and, for $M\subset\N$, we set $\mathsf Z_M(a)=\lbrace x\in\mathsf Z(a)\mid |x|\in M\rbrace$. If the set is a singleton, say $M=\lbrace k\rbrace$, then we write $\mathsf Z_{\lbrace k\rbrace}(a)=\mathsf Z_k(a)$.

\begin{definition}
Let $H$ be an atomic monoid and let $a\in H$.
\begin{enumerate}
\item Factorizations $z_0,\ldots,z_n\in\mathsf Z(a)$ with $n\in\N$ and $\mathsf d(z_{i-1},z_i)\leq N$ for some $N\in\N$ and $i\in [1,n]$ are called
\begin{itemize}
\item an \emph{$N$-chain} concatenating $z_0$ and $z_n$ (in $\mathsf Z(H)$).
\item a \emph{monotone $N$-chain} concatenating $z_0$ and $z_n$ (in $\mathsf Z(H)$) if $|z_{i-1}|\leq |z_i|$ for all $i\in [1,n]$.
\item an \emph{equal-length $N$-chain} concatenating $z_0$ and $z_n$ (in $\mathsf Z(H)$) if $|z_{i-1}|=|z_i|$ for all $i\in [1,n]$.
\end{itemize}
\item The
\begin{itemize}
\item \emph{catenary degree $\cat(a)$}
\item \emph{monotone catenary degree $\cmon(a)$}
\item \emph{equal catenary degree $\ceq(a)$}
\end{itemize}
denotes the smallest $N\in\mathsf N_0\cup\lbrace\infty\rbrace$ such that, for all $z,\,z'\in\mathsf Z(a)$, there is
\begin{itemize}
\item an $N$-chain concatenating $z$ and $z'$.
\item a monotone $N$-chain concatenating $z$ and $z'$.
\item an equal-length $N$-chain concatenating $z$ and $z'$.
\end{itemize}
Then we call
\begin{itemize}
\item $\cat(H)=\sup\lbrace\cat(a)\mid a\in H\rbrace$ the \emph{catenary degree} of $H$.
\item $\cmon(H)=\sup\lbrace\cmon(a)\mid a\in H\rbrace$ the \emph{monotone catenary degree} of $H$.
\item $\ceq(H)=\sup\lbrace\ceq(a)\mid a\in H\rbrace$ the \emph{equal-length catenary degree} of $H$.
\end{itemize}
\end{enumerate}
\end{definition}

Note that $\sup\lbrace\cat(H),\ceq(H)\rbrace\leq\cmon(H)$.

For the description and computation of the monotone catenary degree, we follow the same two step procedure as in \cite{garger10}. In order to formulate this precisely, we need to define another variant of the catenary degree.

\begin{definition}
\label{def:cadj}
Let $H$ be an atomic monoid.
For $a\in H$, we define
 \[
 \cadj (a)=\sup\lbrace\mathsf d(\mathsf Z_k(a),\mathsf Z_l(a))\mid k,\,l\in\mathsf L(a)\mbox{ are adjacent}\rbrace
 \]
 as the \emph{adjacent catenary degree} of $a$.\\
 Also, $\cadj(H)=\sup\lbrace\cadj(a)\mid a\in H\rbrace$ is called the \emph{adjacent catenary degree} of $H$.
\end{definition}

By \cite[(4.1)]{garger10}, we find
\[
\mathsf c(H)\leq\cmon(H)=\sup\lbrace\ceq(H),\,\cadj(H)\rbrace.
\]

Here we follow the same strategy as in \cite[Section 3]{phil10} for the definition of the $\mathcal R$-relation and the $\mu$-invariant.

\begin{definition}
Let $H$ be an atomic monoid and let $a\in H$.
\begin{enumerate}
\item Factorizations $z_0,\ldots,z_n\in\mathsf Z(a)$ with $n\in\N$ and $\gcd(z_{i-1},z_i)\neq 1$ for all $i\in [1,n]$ are called
\begin{itemize}
\item an \emph{$\mathcal R$-chain} concatenating $z_0$ and $z_n$ (in $\mathsf Z(H)$).
\item a \emph{monotone $\mathcal R$-chain} concatenating $z_0$ and $z_n$ (in $\mathsf Z(H)$) if $|z_{i-1}|\leq|z_i|$ for all $i\in [1,n]$.
\item an \emph{equal-length $\mathcal R$-chain} concatenating $z_0$ and $z_n$ (in $\mathsf Z(H)$) if $|z_{i-1}|=|z_i|$ for all $i\in [1,n]$.
\end{itemize}
\item Two elements $z,\,z'\in\mathsf Z(H)$ are
\begin{itemize}
\item \emph{$\mathcal R$-related}
\item \emph{$\Req$-related}
\end{itemize}
if there is an
\begin{itemize}
\item $\mathcal R$-chain
\item equal-length $\mathcal R$-chain
\end{itemize}
 concatenating $z$ and $z'$. We then write $z\approx z'$ respectively $z\equ z'$.
\end{enumerate}
\end{definition}

Note that, with the above definitions, $\approx$ and $\equ$ are congruences on $\mathsf Z(H)\times\mathsf Z(H)$.

Based on these definitions, we can now recall the definition of the $\mu$-invariant (for reference see \cite[Section 3]{phil10}) and give the definition of the $\mueq$-invariant and the $\muadj$-invariant. Note that the definition of the last one differs significantly from the other two since there is no appropriate equivalence relation we can make use of.

\begin{definition}
\label{def:muinv}
Let $H$ be an atomic monoid and let $a\in H$.
\begin{enumerate}
\item \label{def:mu} Let $\mathcal R_a$ denote the set of \emph{$\mathcal R$-equivalence classes} of $\mathsf Z(a)$ and, for $\rho\in\mathcal R_a$, let $|\rho|=\min\lbrace|z|\mid z\in\rho\rbrace$. We set
\[
\mu(a)=\sup\lbrace |\rho|\mid\rho\in\mathcal R_a\rbrace\leq\sup\mathsf L(a)
\]
and define $\mu (H)=\sup\lbrace\mu (a)\mid a\in H\rbrace$.
\item \label{def:mueq} For $k\in\mathsf L(a)$, let $\mathcal R_{a,k}$ denote the set of \emph{$\Req$-equivalence classes} of $\mathsf Z_k(a)$. We set
\[
\mueq (a)=\sup\lbrace k\in\mathsf L(a)\mid |\mathcal R_{a,k}|>1\rbrace\leq\sup\mathsf L(a)
\]
and define $\mueq (H)=\sup\lbrace\mueq (a)\mid a\in H\rbrace$.
\item \label{def:muadj} We set
\[
\muadj (a)=\sup\lbrace k\in\mathsf L(a)\mid \mathsf d(\mathsf Z_k(a),\mathsf Z_l(a))=k\mbox{ for }l\in\mathsf L(a),\,l<k,\,l\mbox{ adjacent to }k\rbrace.
\]
Then we set $\muadj (H)=\sup\lbrace\muadj (a)\mid a\in H\rbrace$.
\end{enumerate}
\end{definition}
Then $\mu (H)=0$ if and only if $|\mathcal R_a|\leq 1$ for all $a\in H$ and $\mueq (H)=0$ if and only if $|\mathcal R_{a,k}|\leq 1$ for all $a\in H$ and $k\in\mathsf L(a)$.

\begin{definition}
Let $H\subset D$ be monoids.
\begin{enumerate}
 \item We call $H\subset D$ \emph{saturated} or, equivalently, \emph{a saturated submonoid} if, for all $a,\,b\in H$, $a\mid b$ in $D$ already implies that $a\mid b$ in $H$; that is, for all $a,\,b\in H$ and $c\in D$, $a=bc$ implies $c\in H$.
 \item If $H\subset D$ is a saturated submonoid, then we set $D/H=\lbrace a\mathsf q(H)\mid a\in D\rbrace$ and $[a]_{D/H}=a\mathsf q(H)$ and we call $\mathsf q(D)/\mathsf q(H)=\mathsf q(D/H)$ the \emph{class group} of $H$ in $D$.
\end{enumerate}
\end{definition}

\begin{definition}
Let $H$ be an atomic monoid. We call
\begin{align*}
\sim_H &=\lbrace(x,y)\in\mathsf Z(H)\times\mathsf Z(H)\mid\pi(x)=\pi(y)\rbrace
&\mbox{the \emph{monoid of relations} of $H$}, \\
\sim_{H,\mathrm{eq}} &=\lbrace (x,y)\in\sim_H\mid |x|=|y|\rbrace
&\mbox{the \emph{monoid of equal-length relations} of $H$}, \\
\sim_{H,\mathrm{mon}} &=\lbrace (x,y)\in\sim_H\mid |x|\leq|y|\rbrace
&\mbox{the \emph{monoid of monotone relations} of $H$},
\end{align*}
and, for $a\in H$, we set
\begin{align*}
\mathcal A_a(\sim_H) &= \mathcal A(\sim_H)\cap(\mathsf Z(a)\times\mathsf Z(a)), \\
\mathcal A_a(\sim_{H,\mathrm{eq}}) &= \mathcal A(\sim_{H,\mathrm{eq}})\cap(\mathsf Z(a)\times\mathsf Z(a)), \\
\mathcal A_a(\sim_{H,\mathrm{mon}}) &= \mathcal A(\sim_{H,\mathrm{mon}})\cap(\mathsf Z(a)\times\mathsf Z(a)).
\end{align*}
\end{definition}

By \cite[Lemma 11]{phil10}, $\sim_H\subset\mathsf Z(H)\times\mathsf Z(H)$ is a saturated submonoid of a free monoid and thus a Krull monoid by \cite[Theorem 2.4.8.1]{MR2194494}.
By \cite[Proposition 4.4.1]{garger10}, $\sim_{H,\mathrm{eq}}\subset\sim_H$ is a saturated submonoid and hence a Krull monoid, and, by \cite[Proposition 4.4.2]{garger10}, $\sim_{H,\mathrm{eq}}$ is finitely generated if $H_{\mathrm{red}}$ is finitely generated. Unfortunately, $\sim_{H,\mathrm{mon}}\subset\sim_H$ is not saturated, but, by \autoref{5.11}, we find that $\sim_{H,\mathrm{mon}}$ is a finitely generated Krull monoid if $H$ is finitely generated.

We briefly recall the main result on the catenary degree from \cite{phil10} and offer a corrected proof for monoids fulfilling the ascending chain condition on principal ideals here.
\begin{lemma}[cf. {\cite[Proposition 8, Corollary 9, Proposition 16]{phil10}}]
\label{2.7}
Let $H$ be an atomic monoid which fulfills the ascending chain condition on principal ideals. Then
\begin{enumerate}
\item\label{2.7.1} $\mathsf c(H)\geq\mu(a)$ for all $a\in H$, and $\mathsf c(H)=\mu(H)$.
\item\label{2.7.2} $\mathsf c(H)=\max\lbrace \mu(a)\mid a\in H,\,\mathcal A_a(\sim_H)\neq\emptyset,\,|\mathcal R_a|>1\rbrace$.
\end{enumerate}
\end{lemma}
\begin{proof}
\mbox{}
\begin{enumerate}
\item First we prove
\[
\mathsf c(a)\geq\mu(a)\mbox{ for all }a\in H.
\]
Let $a\in H$ be such that $|\mathcal R_a|>1$. We may assume that $\mathsf c(a)<\infty$. Let $N\in\N_0$ be such that $\mu(a)\geq N$. Let $\rho\in\mathcal R_a$ be such that $|\rho|\geq N$ and $z\in\rho$ such that $|z|=|\rho|$. Let $z'\in\mathsf Z(a)$ be such that $z\not\approx z'$ and let $z=z_0,z_1,\ldots,z_k=z'$ be a $\mathsf c(a)$-chain concatenating $z$ and $z'$. Let $i\in[1,k]$ be minimal such that $z\not\approx z_i$. Then $z_{i-1}\not\approx z_i$, and therefore
\[
N\leq |z_0|\leq |z_{i-1}|\leq\mathsf d(z_i,z_{i-1})\leq\mathsf c(a).
\]
Till now we have $\mathsf c(H)\geq\mu(H)$. Next we show
\[
\mu(H)\geq\mathsf c(H).
\]
We show that, for all $a\in H$, we have $\mathsf c(a)\leq\mu(H)$. We proceed by induction on $a$. For $a=1$, this is trivial. Now suppose $a\neq 1$ and that, for all $b\in H$ with $b\mid a$, we have $\mathsf c(b)\leq\mu(H)$. Now let $z,\,z'\in\mathsf Z(a)$. If $z\not\approx z'$, then there are $z'',\,z'''\in\mathsf Z(a)$ such that $z''\approx z,\,z'''\approx z'$, and $z''$ and $z'''$ are minimal in their $\mathcal R$-classes with respect to their lengths. Since $\gcd(z'',z''')=1$, we find $\mathsf d(z'',z''')=\max\lbrace |z''|,|z'''|\rbrace\leq\mu(a)\leq\mu(H)$. Now it remains to show that, for any two factorizations $z,\,z'\in\mathsf Z(a)$ with $z\approx z'$, there is a $\mu(H)$-chain concatenating them. By definition, there is an $\mathcal R$-chain $z_0,\ldots,z_k$ with $z=z_0$ and $z'=z_k$, and $g_i=\gcd(z_{i-1},z_i)\neq 1$ for all $i\in [1,k]$. Since $\pi_H(g_i^{-1}z_{i-1})\mid a$, we find a $\mu(H)$-chain concatenating $g_i^{-1}z_{i-1}$ and $g_i^{-1}z_i$ for all $i\in [1,k]$ by induction hypothesis, and thus there is a $\mu(H)$-chain concatenating $z_{i-1}$ and $z_i$ for all $i\in [1,k]$; thus there is a $\mu(H)$-chain concatenating $z$ and $z'$. So $\mathsf c(a)\leq\mu(H)$.
\item When we compare the definitions, we see that the only thing remaining is
\[
\lbrace a\in H\mid\mathcal A_a(\sim_H)\neq\emptyset,\,|\mathcal R_a|>1\rbrace=\lbrace a\in H\mid|\mathcal R_a|>1\rbrace. 
\]
One inclusion is trivial and, for the other one, let $a\in H$ be such that $|\mathcal R_a|>1$, and let $z,\,z'\in\mathsf Z(a)$ be two factorizations of $a$ such that $z\not\approx z'$ and such that both are minimal in their $\mathcal R$-equivalence classes with respect to their lengths. Now assume $(z,z')\notin\mathcal A(\sim_H)$. Then there is $k\geq 2$ and $(x_1,y_1),\ldots,(x_k,y_k)\in\mathcal A(\sim_H)$ such that $(z,z')=(x_1,y_1)\cdot\ldots\cdot(x_k,y_k)$. But now we find the following $\mathcal R$-chain from $z$ to $z'$: $z_0=z$ and $z_i=z_{i-1}x_i^{-1}y_i$ for $i\in[1,k]$. Then $z_k=z'$ and $\gcd(z_{i-1},z_i)\neq 1$. Since this is a contradiction, we have $(z,z')\in\mathcal A(\sim_H)$, and thus $(z,z')\in\mathcal A_a(\sim_H)\neq\emptyset$.
\qedhere
\end{enumerate}
\end{proof}

\bigskip
\section{A characterization of the monotone catenary degree by monoids of relations}
\label{sec:cmon}
\bigskip

\begin{lemma}
\label{3.1}
Let $H$ be an atomic monoid, $a\in H$ and $x,\,y\in\mathsf Z(a)$.
\begin{enumerate}
\item \label{3.1.1} If $x\not\approx_{\mathrm{eq}} y$, then $(x,y)\in\mathcal A_a(\sim_{H,\mathrm{eq}})$.
\item \label{3.1.2} Let $k,\,l\in\mathsf L(a)$ be adjacent with $k<l$. If $\mathsf d(\mathsf Z_k(a),\mathsf Z_l(a))=l$, then $(x,y)\in\mathcal A_a(\sim_{H,\mathrm{mon}})$ for all $x\in\mathsf Z_k(a)$ and $y\in\mathsf Z_l(a)$.
\end{enumerate}
\end{lemma}
\begin{proof}
Since the arithmetic of $H$ is determined solely by $H_{\mathrm{red}}$, we may assume that $H$ is reduced.
\begin{enumerate}
\item Let $a\in H$ and $x,\,y\in\mathsf Z(a)$ be such that $(x,y)\notin\mathcal A_a(\sim_{H,\mathrm{eq}})$. Then, trivially, $(x,y)\notin\mathcal A(\sim_{H,\mathrm{eq}})$ and thus there are $(x_1,y_1),\ldots,(x_k,y_k)\in\mathcal A(\sim_{H,\mathrm{eq}})$ with $k\geq 2$ such that $(x,y)=(x_1,y_1)\mdots(x_k,y_k)$. Then $x=x_1\mdots x_k,\,y_1x_2\mdots x_k,\,y_1\mdots y_k=y$ is an $\mathcal R_{\mathrm{eq}}$-chain concatenating $x$ and $y$, and therefore $x\equ y$.
\item Let $a\in H$, let $k,\,l\in\mathsf L(a)$ be adjacent with $k<l$ and $\mathsf d(Z_k(a),Z_l(a))=l$, and let $x\in\mathsf Z_k(a)$ and $y\in\mathsf Z_l(a)$. Now suppose $(x,y)\notin\mathcal A_a(\sim_{H,\mathrm{mon}})$. Then, trivially, $(x,y)\notin\mathcal A(\sim_{H,\mathrm{mon}})$ and there are $(x_1,y_1),\ldots,(x_k,y_k)\in\mathcal A(\sim_{H,\mathrm{mon}})$ with $k\geq 2$ and $|y_1|-|x_1|\leq\ldots\leq |y_k|-|x_k|$. Then we set $x'=x_1^{-1}y_1x$. If $|y_1|-|x_1|=0$, we find $|x'|=k$ and $\gcd(x',y)\neq 1$, a contradiction to $\mathsf d(\mathsf Z_k(a),\mathsf Z_l(a))=l$. Otherwise, if $|y_1|-|x_1|>0$, then $k=|x|<|x'|<|y|=l$, a contradiction to $k$ and $l$ being adjacent.
\qedhere
\end{enumerate}
\end{proof}

In principal, we follow the same strategy as in \cite[Section 3]{phil10} for the $\mu$-invariant when studying the $\mueq$-invariant. For the $\muadj$-invariant, we cannot construct an equivalence relation like the $\mathcal R$-relation or the $\Req$-relation. Thus we follow a slightly modified strategy in the proofs of parts \ref{3.2.3} and \ref{3.2.4} from \autoref{3.2}.

\begin{theorem}
\label{3.2}
Let $H$ be an atomic monoid. Then
\begin{enumerate}
\item \label{3.2.1} $\ceq (a)\geq\mueq(a)$ for all $a\in H$, and $\ceq(H)=\mueq(H)$.
\item \label{3.2.2} $\ceq (H)=\sup\lbrace\mueq (a)\mid a\in H,\,\mathcal A_a(\sim_{H,\mathrm{eq}})\neq\emptyset,\,|\mathcal R_{a,k}|>1\mbox{ for some }k\in\mathsf L(a)\rbrace\\
\mbox{ }\quad\quad\;\;=\sup\lbrace k\in\N\mid a\in H,\,\mathcal A_a(\sim_{H,\mathrm{eq}})\neq\emptyset,\,k\in\mathsf L(a),\,|\mathcal R_{a,k}|>1\rbrace$.
\item \label{3.2.3} $\cadj(a)\geq\muadj(a)$ for all $a\in H$, and $\cadj(H)=\muadj(H)$.
\item \label{3.2.4} $\cadj (H)=\sup\lbrace\muadj(a)\mid a\in H,\,\mathcal A_a(\sim_{H,\mathrm{mon}})\neq\emptyset\rbrace$.
\end{enumerate}
In particular,
\begin{multline*}
\cmon(H)=\sup(\lbrace\mueq(a)\mid a\in H,\,\mathcal A_a(\sim_{H,\mathrm{eq}}),\,|\mathcal R_{a,k}|>1\mbox{ for some }k\in\mathsf L(a)\rbrace \\ \cup\lbrace\muadj(a)\mid a\in H,\,\mathcal A_a(\sim_{H,\mathrm{mon}})\rbrace).
\end{multline*}
\end{theorem}
\begin{proof}
Since the arithmetic of $H$ is determined solely by $H_{\mathrm{red}}$ we may assume that $H$ is reduced.
\begin{enumerate}
\item First we prove $\ceq(a)\geq\mueq(a)$ for all $a\in H$. We may assume that $\ceq(a)<\infty$ and $\mueq(a)\geq 1$. Let $N\in\N$ be such that $N\leq\mueq(a)$. Then there exists $k\in\mathsf L(a)$ such that $|\mathcal R_{a,k}|>1$ and $k\geq N$. Let $z,\,z'\in\mathsf Z_k(a)$ be such that $z\not\equ z'$, and let $z=z_0,z_1,\ldots,z_n=z'$ be a $\ceq(a)$-equal-length chain concatenating $z$ and $z'$. Now we choose $i\in[0,n-1]$ minimal such that $z\not\equ z_i$. Then $z_{i-1}\not\equ z_i$, and we find
\[
\ceq(a)\geq\mathsf d(z_{i-1},z_i)=k\geq N.
\]
Now we prove $\mueq (H)\geq\ceq (H)$. We show that, for all $N\in\N_0$, all $a\in H$, and all factorizations $z,\,z'\in\mathsf Z(a)$ with $|z|=|z'|\leq N$, there is a $\mueq (H)$-equal-length-chain from $z$ to $z'$. We proceed by induction on $N$. If $N=0$, then $z=z'=1$ and $\mathsf d(z,z')=0\leq\mueq (H)$. Suppose $N\geq 1$ and that, for all $a\in H$ and all $z,\,z'\in\mathsf Z(a)$ with $|z|=|z'|<N$, there is a $\mueq (H)$-equal-length-chain from $z$ to $z'$. Now let $a\in H$ and let $z,\,z'\in\mathsf Z(a)$ with $|z|=|z'|\leq N$. If $z\not\equ z'$, then $\mueq (H)\geq\mueq (a)\geq |z|=\mathsf d(z,z')$. Now it remains to show that, for any two factorizations $z,\,z'\in\mathsf Z(a)$ with $|z|=|z'|\leq N$ and $z\equ z'$, there is a $\mueq (H)$-equal-length-chain concatenating them. By definition, there is an $\Req$-chain $z_0,\ldots,z_k$ with $z_0=z$, $z'=z_k$, $g_i=\gcd(z_{i-1},z_i)\neq 1$, and $|z_i|=|z|$ for all $i\in[1,k]$. By induction hypothesis, there is a $\mueq (H)$-equal-length-chain from $g_i^{-1}z_{i-1}$ to $g_i^{-1}z_i$ for all $i\in[1,k]$, and hence there is a $\mueq (H)$-equal-length-chain from $z_{i-1}$ to $z_i$ for $i\in[1,k]$; thus there is a $\mueq (H)$-equal-length chain from $z$ to $z'$.
\item By part~\ref{3.2.1}, we have $\ceq(H)=\mueq(H)$ and, by \autoref{def:muinv}.\ref{def:mueq}, the third equality is obvious. Thus it suffices to show that
\begin{multline*}
\lbrace\mueq (a)\mid a\in H,\,|\mathcal R_{a,k}|>1\mbox{ for some }k\in\mathsf L(a)\rbrace=\\
\lbrace\mueq (a)\mid a\in H,\,\mathcal A_a(\sim_{H,\mathrm{eq}})\neq\emptyset,\,|\mathcal R_{a,k}|>1\mbox{ for some }k\in\mathsf L(a)\rbrace.
\end{multline*}
The inclusion from right to left is clear. Now let $a\in H$ and $k\in\mathsf L(a)$ be such that $|\mathcal R_{a,k}|>1$. Then there exist $z,\,z'\in\mathsf Z_k(a)$ such that $z\not\equ z'$. By \autoref{3.1}.\ref{3.1.1}, we find $(z,z')\in\mathcal A_a(\sim_{H,\mathrm{eq}})\neq\emptyset$.
\item First let $a\in H$. We show that $\cadj (a)\geq\muadj(a)$, and then $\cadj(H)\geq\muadj(H)$ follows by passing to the supremum on both sides. If $\muadj (a)=0$ or $\muadj (a)=\infty$, this is trivial. Now let $\muadj (a)=l\in\N$. Then there is $k\in\mathsf L(a)$ and $k<l$ with $l$ adjacent to $k$. Then, by \autoref{def:cadj}, $\cadj (a)\geq\mathsf d(\mathsf Z_k(a),\mathsf Z_l(a))=\max\lbrace k,l\rbrace=l=\muadj(a)$.\\
Now we prove $\muadj(H)\geq\cadj(H)$. We must prove that $\cadj(a)\leq\muadj(H)$ for all $a\in H$. Assume to the contrary that there is some $a\in H$ such that $\cadj(a)>\muadj(H)$. Let $l\in\N$ be minimal such that there is some $k<l$ and $a\in H$ with $k$ and $l$ adjacent lengths of $a$ and $\cadj(a)=\mathsf d(\mathsf Z_k(a),\mathsf Z_l(a))$. If $\mathsf d(\mathsf Z_k(a),\mathsf Z_l(a))<l$, then there are some $x\in\mathsf Z_k(a)$ and $y\in\mathsf Z_l(a)$ such that $g=\gcd(x,y)\neq 1$. If $b=\pi_H(g^{-1}x)$, then $k-|g|$ and $l-|g|$ are adjacent lengths of $b$ and
\[
\cadj(a)=\mathsf d(\mathsf Z_k(a),\mathsf Z_l(a))\leq\mathsf d(\mathsf Z_{k-|g|}(b),\mathsf Z_{l-|g|}(b))\leq\cadj(b),
\]
and by the minimal choice of $l$, we infer that $\cadj(b)\leq\muadj(H)$, a contradiction.
\item By part~\ref{3.2.3} and \autoref{def:muinv}.\ref{def:muadj}, we find
\[
\cadj (H)=\muadj (H)=\sup\lbrace\muadj (a)\mid a\in H\rbrace.
\]
Thus it suffices to show that
\[
\sup\lbrace\muadj(a)\mid a\in H\rbrace=\sup\lbrace\muadj(a)\mid a\in H,\,\mathcal A_a(\sim_{H,\mathrm{mon}})\rbrace.
\]
In fact, we only have to show that $\sup\lbrace\muadj(a)\mid a\in H\rbrace\leq\sup\lbrace\muadj(a)\mid a\in H,\,\mathcal A_a(\sim_{H,\mathrm{mon}})\rbrace$. Now let $a\in H$ and $\muadj(a)=k\in\N$. Then there is $l\in\mathsf L(a)$ with $l<k$, $l$ adjacent to $k$, and $\mathsf d(\mathsf Z_k(a),\mathsf Z_l(a))=k$. Now let $x\in\mathsf Z_l(a)$ and $y\in\mathsf Z_k(a)$. Then we have $\gcd(x,y)=1$. By \autoref{3.1}.\ref{3.1.2}, we have $(x,y)\in\mathcal A_a(\sim_{H,\mathrm{mon}})\neq\emptyset$.
\end{enumerate}
The additional statement now follows easily by parts~\ref{3.2.2} and \ref{3.2.4}.
\end{proof}

\begin{lemma}
\label{3.3}
Let $H$ be an atomic monoid and let $a\in H$.
\begin{enumerate}
\item \label{3.3.1} Let $x,\,y\in\mathsf Z(a)$ with $\min\lbrace |x|,|y|\rbrace>\cmon(a)$. Then there is a monotone $\mathcal R$-chain concatenating $x$ and $y$, and thus $x\approx y$; in particular, if $|x|=|y|$, then $x\approx_{\mathrm{eq}} y$.
\item \label{3.3.2} Let $k,\,l\in\mathsf L(a)$. Then
\[
\mathsf d(\mathsf Z_k(a),\mathsf Z_l(a))=\max\lbrace k,l\rbrace
\quad\mbox{if and only if}\quad
\gcd(x,y)=1\mbox{ for all }x\in\mathsf Z_k(a)\mbox{ and }y\in\mathsf Z_l(a).
\]
\item \label{3.3.3} Let $k,\,l\in\mathsf L(a)$ be adjacent with $k<l$ such that there are $x\in\mathsf Z_k(a)$ and $y\in\mathsf Z_l(a)$ such that there is a monotone $\mathcal R$-chain concatenating $x$ and $y$. Then $\muadj(a)\neq 1$.
\end{enumerate}
\end{lemma}
\begin{proof}
Since the arithmetic of $H$ is determined solely by $H_{\mathrm{red}}$, we may assume that $H$ is reduced.
\begin{enumerate}
\item Let $a\in H$ and $x,\,y\in\mathsf Z(a)$ be such that $\min\lbrace |x|,|y|\rbrace>\cmon(a)$. We may assume that $|x|\leq |y|$. Then there is a monotone $\cmon(a)$-chain concatenating $x$ and $y$, say $z_0=x,z_1,\ldots,z_k=y$. Since, for all $i\in[1,k]$, we have $\mathsf d(z_{i-1},z_i)\leq\cmon(a)<|x|=|z_0|$, we have $\gcd(z_{i-1},z_i)\neq 1$ for all $i\in[1,k]$. Thus $z_0,\ldots,z_k$ is a monotone $\mathcal R$-chain concatenating $x$ and $y$, and therefore $x\approx y$. If $|x|=|y|$, then $z_0,\ldots,z_k$ is an equal-length chain, and therefore $x\equ y$.
\item Follows immediately by the definition of the distance of factorizations in $\mathsf Z(H)$.
\item Let $a\in H$, let $k,\,l\in\mathsf L(a)$ be adjacent with $k<l$, let $x\in\mathsf Z_k(a)$, and $y\in\mathsf Z_l(a)$ be such that there is a monotone $\mathcal R$-chain from $x$ to $y$, say $z_0=x,z_1,\ldots,z_n=y$ for some $n\in\N$. Now choose $i\in[1,n]$ minimal such that $|z_i|=l$. Due to the minimality of $i$, we find $z_{i-1}\in\mathsf Z_k(a)$. Since $\gcd(z_{i-1},z_i)\neq 1$, we find $\mathsf d(\mathsf Z_k(a),\mathsf Z_l(a))<l$, and therefore $\muadj(a)\neq l$.
\qedhere
\end{enumerate}
\end{proof}

\begin{lemma}
\label{3.4}
Let $H$ be an atomic monoid. Then
\begin{enumerate}
\item\label{3.4.1} $\ceq(H)\leq\sup\lbrace |y|\mid(x,y)\in\mathcal A(\sim_{H,\mathrm{eq}}),\,x\not\equ y\rbrace$.
\item\label{3.4.2} $\cadj(H)\leq\sup\lbrace |y|\mid(x,y)\in\mathcal A(\sim_{H,\mathrm{mon}}),\,|x|<|y|,\;|x|,|y|\in\mathsf L(\pi_H(x))\mbox{ adjacent, and there is no}$ $\mbox{monotone $\mathcal R$-chain from $x$ to $y$}\rbrace$.
\end{enumerate}
In particular,
\begin{multline*}
\cmon(H)\leq\sup \lbrace |y|\mid(x,y)\in\mathcal A(\sim_{H,\mathrm{mon}}),\mbox{ there is no monotone $\mathcal R$-chain from }x\mbox{ to }y,\\
\mbox{ and either }|x|=|y|\mbox{ or }|x|,\,|y|\in\mathsf L(\pi_H(x))\mbox{ are adjacent}\rbrace.
\end{multline*}
\end{lemma}
\begin{proof}
Since the arithmetic of $H$ is determined solely by $H_{\mathrm{red}}$, we may assume that $H$ is reduced.
\begin{enumerate}
\item The inequality $\ceq(H)\leq\sup\lbrace |y|\mid(x,y)\in\mathcal A(\sim_{H,\mathrm{eq}})$ has been proven in \cite[Proposition 4.4.3]{garger10}. The slightly stronger statement here follows immediately by the definition of $\mueq(\cdot)$; see \autoref{def:muinv}.\ref{def:mueq}.
\item By \autoref{3.2}.\ref{3.2.4}, we have $\cadj(H)=\sup\lbrace\muadj(a)\mid a\in H,\,\mathcal A_a(\sim_{H,\mathrm{mon}})\neq\emptyset\rbrace$. Now the assertion follows from \autoref{3.3}.\ref{3.3.3}, \autoref{3.1}.\ref{3.1.2}, and the definition of $\muadj(\cdot)$; see \autoref{def:muinv}.\ref{def:muadj}.
\end{enumerate}
The additional statement follows from
\[
\cmon(H)=\sup\lbrace\ceq(H),\cadj(H)\rbrace
\quad\mbox{and}\quad
\mathcal A(\sim_{H,\mathrm{eq}})\subset\mathcal A(\sim_{H,\mathrm{mon}}).
\qedhere
\]
\end{proof}

\bigskip
\section{Tameness and monotone chains}
\label{sec:mstep}
\bigskip

\begin{definition}
\label{def-tame}
Let $H$ be an atomic monoid.
\begin{enumerate}
\item For $a\in H$ and $x\in\mathsf Z(H)$, let $\mathsf t(a,x)$ denote the smallest $N\in\mathbb N_0\cup\lbrace\infty\rbrace$ with the following property:
\begin{itemize}
 \item[] If $\mathsf Z(a)\cap x\mathsf Z(H)\neq\emptyset$ and $z\in\mathsf Z(a)$, then there exists some $z'\in\mathsf Z(a)\cap x\mathsf Z(H)$ such that $\mathsf d(z,z')\leq N$.
\end{itemize}
For subsets $H'\subset H$ and $X\subset\mathsf Z(H)$, we define
\[
 \mathsf t(H',X)=\sup\lbrace\mathsf t(a,x)\mid a\in H',\,x\in X\rbrace,
\]
and we define $\mathsf t(H)=\mathsf t(H,\mathcal A(H_{\mathrm{red}}))$. This is called the \emph{tame degree} of $H$.
\item If $\mathsf t(H)<\infty$, then we call $H$ \emph{tame}.
\end{enumerate}
\end{definition}

Being tame is a very strong finiteness condition within non-unique factorization theory, in particular, the finiteness of the tame degree implies the finiteness of the elasticity and the catenary degree among other invariants. Next, we give a list of examples where tameness is characterized in various classes of monoids and domains; for a similar list, the reader is referred to \cite[Examples 3.2]{MR2660909}.

\begin{enumerate}
\item \emph{Finitely generated monoids}. If $H_{\mathrm{red}}$ is finitely generated, then $H$ is tame (see \cite[Theorem 3.1.4]{MR2194494}).
\item \emph{Finitely primary monoids}. Let $H$ be finitely primary of rank $s\in\N$. Then $H$ is tame if and only if $s=1$ (see \cite[Theorem 3.1.5]{MR2194494}).
\item \emph{Weakly Krull domains}. Let $R$ be a $v$-noetherian weakly Krull domain with nonzero conductor $\mathfrak f=(R:\widehat R)$ and finite $v$-class group $\mathcal C_v(R)$. Note that, in particular, orders in algebraic number fields fulfill all these properties.\\
Then $R$ is tame if and only if, for every nonzero prime ideal $\mathfrak p\in\mathfrak X(R)$ with $\mathfrak p\supset\mathfrak f$, there is precisely one $\mathfrak P\in\mathfrak X(\widehat R)$ such that $\mathfrak P\cap R=\mathfrak p$ (see \cite[Theorem 3.7.1]{MR2194494}).
\item \emph{Krull monoids} and therefore \emph{Krull domains}. Let $H$ be a Krull monoid, $F=\mathcal F(P)$ a monoid of divisors and $G_P=\lbrace [p]\mid p\in P\rbrace\subset F/H_{\mathrm{red}}=G$ the set of classes containing prime divisors. Suppose that one of the following conditions hold:
\begin{enumerate}
\item $H$ has the approximation property.
\item Every $g\in G_P$ contains at least two prime divisors.
\item There is an $m\in\N$ such that $-G_p\subset m(G_P\cup\lbrace 0\rbrace)$.
\item The torsion free rank of $G$ is finite.
\end{enumerate}
Then $H$ is tame if and only if $\mathsf D(G_P)<\infty$ (see \cite[Theorem 4.2]{MR2660909}). In particular, all principal orders in algebraic number fields are tame.
\item \emph{$C$-like monoids}. Let $H$ be a $C$-like monoid. Then $H$ is tame if and only if the natural map $s$-$\spec(\widehat H)\rightarrow s$-$\spec(H)$ is bijective (see \cite[Theorem 8.3]{Ka11} and \cite[Definition 5.6]{Ka11} for a precise definition of $C$-like monoids).\\
Next we give two examples of $C$-like monoids. $\punkt R$ is a $C$-like monoid if
\begin{itemize}
 \item (see \cite[Proposition 6.1]{Ka11}) $R$ is an integral domain and $\punkt R$ is finitely primary.
 \item (see \cite[Proposition 6.5]{Ka11}) $R$ is a Mori domain with complete integral closure $\widehat R$, $\mathcal C_v(\widehat R)$ is finite, $(R:\widehat R)\neq 0$, and either
\begin{itemize}
 \item  $R$ is semilocal, and $\widehat R/(R:\widehat R)$ is quasi artinian or
 \item $\mathcal C_v(R)$ is finite and $S^{-1}\widehat R/S^{-1}(R:\widehat R)$ is quasi artinian, where $S\subset\punkt R$ is the submonoid of regular elements.
\end{itemize}
\end{itemize}
\end{enumerate}

While the tameness of a monoid implies the finiteness of the catenary degree, it does not imply the finiteness of the equal catenary degree and therefore not the finiteness of the monotone catenary degree. In order to point this out, we discuss a monoid originally introduced in \cite[Example 4.5]{MR2387746}.

Recall that a monoid $H$ is called \emph{finitely primary} if there exist $s,\,k\in\N$ and a factorial monoid $F=[p_1,\ldots,p_s]\times\mal F$ with the following properties:
\begin{itemize}
 \item $H\setminus\mal H\subset p_1\mdots p_s F$ and
 \item $(p_1\mdots p_s)^kF\subset H$.
\end{itemize}
If this is the case, then we call $H$ a \emph{finitely primary monoid} of \emph{rank} $s$ and \emph{exponent} $k$.

\begin{example}[cf. {\cite[Example 4.5]{MR2387746}}]
\label{4.2}
There exists a tame monoid $H$ such that $\ceq(H)=\infty$, and thus $\cmon(H)=\infty$ but $\cadj(H)<\infty$.
\end{example}
\begin{proof}
We proceed in four steps.
\begin{enumerate}
 \item We start with a construction which was first used in \cite{phd-hassler}. Let $G$ be an additively written abelian group and $f:G\rightarrow N_0$ a map with $f(0)=0$ and finite image $f(G)$ such that, for all $g,\,g'\in G$, the following two conditions are satisfied:
\begin{enumerate}
 \item $f(g+g')\leq f(g)+f(g')$ and
 \item if $f(g)=0$, then $f(-g)=0$.
\end{enumerate}
 Then, by construction,
\[
 H(G,f)=\lbrace (g,k)\mid g\in G,\,k\in\N_0\,\mbox{with}\,k\geq f(g)\rbrace\subset(G\times\N_0,+)
\]
 is a finitely primary monoid of rank one and exponent $\max f(G)$.
 \item We consider a group $G$ with basis $E=\lbrace e_m,e_m'\mid m\in\N\rbrace$, where $\ord(e_m)=\ord(e'_m)=m$ for all $m\in\N$, whence
\[
 G=\bigoplus_{m\in\N}(\langle e_m\rangle\oplus\langle e'_m\rangle)=\bigoplus_{m\in\N}(\Z/m\Z)^2.
\]
 Let $f:G\rightarrow\N_0$ be defined by $f(0)=0$, $f(e)=1$ for all $e\in E$, and $f(g)=2$ for all $g\in G\setminus (E\cup\lbrace 0\rbrace)$. Then $f$ satisfies all properties required in part 1, and we study $H(G,f)=H$.
 \item Let $n\in\N$ and $a_n=(0,n)\in H$. Then $z_n=(e_n,1)+\ldots+(e_n,1)\in\mathsf Z_n(a_n)$, $z_n'=(e'_n,1)+\ldots+(e'_n,1)\in\mathsf Z_n(a_n)$, and we assert that, for every $z\in\mathsf Z_n(a_n)\setminus\lbrace z_n\rbrace$, we have $\mathsf d(z_n,z)=n$. Then we find
\[
 \cmon(H)\geq\ceq(H)\geq\ceq(a_n)\geq n,
\]
 whence $\cmon(H)=\ceq(H)=\infty$. Let $z\in\mathsf Z_n(a_n)\setminus\lbrace z_n\rbrace$. Then $z$ has the form $z=(g_1,1)+\ldots+(g_n,1)$ with $g_1,\ldots,g_n\in G$. Since $1\geq f(g_i)$ for every $i\in [1,n]$, it follows that $\lbrace g_1,\ldots,g_n\rbrace\subset E\cup\lbrace 0\rbrace$. If $e_n\in\lbrace g_1,\ldots,g_n\rbrace$, then $g_1=\ldots=g_n=e_n$, because $E$ is a basis. Since $z\neq z_n$, we infer that $e_n\notin\lbrace g_1,\ldots,g_n\rbrace$, whence $\mathsf d(z_n,z)=n$.
 \item Let $(g,n)\in H(G,f)$. First we prove that either $\max\mathsf L((g,n))=n-1$ or $\mathsf d(\mathsf Z_{n-1}((g,n)),\mathsf Z_n((g,n)))=3$.
If $\max\mathsf L((g,n))=n-1$, then the assertion is trivial. Thus assume $\max\mathsf L((g,n))=n$. Then there is $z\in\mathsf Z_n((g,n))$ of the form $z=(e_{n_1},1)+\ldots+(e_{n_n},1)$ with $e_{n_1},\ldots,e_{n_n}\in E$ and $e_{n_1}+\ldots+e_{n_n}=g$. Now we find $z'=(e_{n_1}+e_{n_2}+e_2,2)+(e_2,1)+(e_{n_3},1)+\ldots+(e_{n_n},1)\in\mathsf Z_{n-1}((g,n))$ and $\mathsf d(\mathsf Z_{n-1}((g,n)),\mathsf Z_n((g,n)))\leq\mathsf d(z,z')=3$. This proves the assertion in the second case.

If $n\leq 6$, then $\max\mathsf L((g,n))\leq 6$, and thus $\cadj((g,n))\leq 6$. Let now $n\geq 7$. Then there are $n'\in[3,6]$ and $n''\in\N$ such that $n=n'+4n''$ and $(g,n)=(g,n')+n''(0,4)$. Since we have $z_1=4(e_4,1),\,z_2=(e_2,1)+(e_4,1)+(e_2+3e_4,2),\,z_3=2(e_2+2e_4,2)\in\mathsf Z((0,4))$, $|z_1|=4,\,|z_2|=3,\,|z_3|=2$, and $\mathsf d(z_1,z_2)=\mathsf d(z_2,z_3)=3$, we find that $\cadj((g,n))\leq\max\lbrace 3,n'\rbrace=n'\leq 6<\infty$.
 \item By \cite[Theorem 3.1.5.2.a]{MR2194494}, each finitely primary monoid of rank one is tame.
\qedhere
\end{enumerate}
\end{proof}

Note that, for the monoid $H$ in \autoref{4.2}, we have $\cadj(H)<\infty$, and therefore the question of whether the finiteness of the tame degree implies the finiteness of the adjacent catenary degree remains open. Nevertheless, the following result from \cite{MR2660909} might be interpreted as a strong sign that the tame degree can dominate the adjacent catenary degree.

\begin{lemma}[cf. {\cite[Theorem 5.1.b]{MR2660909}}]
Let $H$ be a tame monoid. Then there exists a constant $M\in\N_0$ such that, for all $a\in H$ and for each two adjacent lengths $k,\,l\in\mathsf L(a)\cap[\min\mathsf L(a)+M,\max\mathsf L(a)-M]$, we have $\mathsf d(\mathsf Z_k(a),\mathsf Z_l(a))\leq M$.
\end{lemma}

\begin{lemma}
\label{4.4}
Let $H$ be an atomic monoid.
\begin{enumerate}
\item\label{4.4.1} If $H$ is half-factorial, then $\cadj(H)=0$ and $\cmon(H)=\ceq(H)=\mathsf c(H)$.
\item\label{4.4.2} If $a\in H$ satisfies $|\mathsf L(a)|\leq 2$, then $\muadj(a)\leq\mathsf t(H)$.
\end{enumerate}
\end{lemma}
\begin{proof}
Since the arithmetic of $H$ is determined solely by $H_{\mathrm{red}}$, we may assume that $H$ is reduced.
\begin{enumerate}
\item Since, for all $a\in H$ with $|\mathsf L(a)|=1$, we have no adjacent lengths, it follows that $\cadj(H)=0$, and thus $\cmon(H)=\ceq(H)$. As---in this special situation---every chain of factorizations is an equal-length chain of factorizations, we get $\ceq(H)=\mathsf c(H)$.
\item Choose $a\in H$ such that $|\mathsf L(a)|\leq 2$. If $|\mathsf L(a)|=1$, then $\muadj(a)=0$. Now suppose $|\mathsf L(a)|=2$. If $\muadj(a)=0$, then there is nothing to show. Now suppose $\muadj(a)>0$. Then $\muadj(a)=\max\mathsf L(a)$, and thus $\gcd(x,y)=1$ for all $x,\,y\in\mathsf Z(a)$ with $|x|=\min\mathsf L(a)$ and $|y|=\max\mathsf L(a)$. Let $x,\,y\in\mathsf Z(a)$ with $|x|=\min\mathsf L(a)$ and $|y|=\max\mathsf L(a)$ and choose $u\in\mathcal A(H)$ such that $x\in\mathsf Z(a)\cap u\mal H\mathsf Z(H)$. Then there is no $y'\in\mathsf Z(a)\cap u\mal H\mathsf Z(H)$ with $|y'|=|y|$. Now we find
\[
\mathsf t(H)\geq\mathsf t(a,u\mal H)\geq\mathsf d(y,\mathsf Z(a)\cap u\mal H\mathsf Z(H))=|y|=\max\mathsf L(a)=\muadj(a).
\qedhere
\]
\end{enumerate}
\end{proof}

Next we formulate another variant of the catenary degree, which is somewhat similar to the adjacent catenary degree and equals it in a special situation. The main difference is that we can prove that the $m$-adjacent catenary degree is finite for tame monoids when $m$ is sufficiently large.

\begin{definition}
\label{def:cm}
Let $H$ be an atomic monoid, let $a\in H$ and let $m\in\N$.
\begin{enumerate}
\item We set
\[
\mu_{\mathrm{ad},m}(a)=\sup\lbrace k\in\mathsf L(a)\mid\mathsf d(\mathsf Z_k(a),\mathsf Z_{[k-m,k)}(a))=k\rbrace
\quad\mbox{and}\quad
\mu_{\mathrm{ad},m}(H)=\sup\lbrace\mu_{\mathrm{ad},m}(a)\mid a\in H\rbrace.
\]
\item\label{def:cm1} We define
\[
\mathsf c_{\mathrm{ad},m}(a)=\sup\lbrace\mathsf d(\mathsf Z_k(a),\mathsf Z_{[k-m,k)}(a))\mid k\in\mathsf L(a)\rbrace
\]
as the \emph{$m$-adjacent catenary degree} of $a$.\\
Also, $\mathsf c_{\mathrm{ad},m}(H)=\sup\lbrace\mathsf c_{\mathrm{ad},m}(a)\mid a\in H\rbrace$ is called the \emph{$m$-adjacent catenary degree} of $H$.
\end{enumerate}
\end{definition}
Obviously, we find
\[
\mathsf c_{\mathrm{ad},m}(H)
\begin{cases}
= 0 & m<\min\triangle(H) \\
\leq\cadj(H) \\
=\cadj(H) & \triangle(H)=\lbrace n\rbrace\mbox{ and }n\leq m<2n.
\end{cases}
\]

Since the definitions of the $m$-adjacent catenary degree and $\mu_{\mathrm{ad},m}(H)$ are similar to those of the adjacent catenary degree and $\muadj(H)$, we can now prove the analog of \autoref{3.2}.\ref{3.2.3} for the two newly defined invariants.

\begin{theorem}
\label{4.6}
Let $H$ be an atomic monoid and let $m\in\N$. Then
\begin{enumerate}
\item \label{4.6.1} $\mathsf c_{\mathrm{ad},m}(a)\geq\mu_{\mathrm{ad},m}(a)$ for all $a\in H$, and $\mathsf c_{\mathrm{ad},m}(H)=\mu_{\mathrm{ad},m}(H)$.
\item \label{4.6.2} $\mathsf c_{\mathrm{ad},m}(H)\leq\mathsf t(H)$ for all $m\geq\mathsf t(H)$.
\end{enumerate}
\end{theorem}
\begin{proof}
\mbox{}
\begin{enumerate}
\item For $m<\min\triangle(H)$, we have $\mathsf c_{\mathrm{ad},m}(H)=0=\mu_{\mathrm{ad},m}(H)$ by definition. Now let $m\in\N$ and $m\geq\min\triangle(H)$.\\
First we let $a\in H$ and show that $\mathsf c_{\mathrm{ad},m}(a)\geq\mu_{\mathrm{ad},m}(a)$, after which $\mathsf c_{\mathrm{ad},m}(H)\geq\mu_{\mathrm{ad},m}(H)$ follows by passing to the supremum on both sides. If $\mu_{\mathrm{ad},m}(a)=0$ or $\mu_{\mathrm{ad},m}(a)=\infty$, this is trivial. Now let $\mu_{\mathrm{ad},m}(a)=k\in\N$ and $[k-m,k)\cap\mathsf L(a)=\lbrace l_1,\ldots,l_n\rbrace$. Then, by \autoref{def:cm}.\ref{def:cm1}, $\mathsf c_{\mathrm{ad},m}(a)\geq\mathsf d(\mathsf Z_k(a),\mathsf Z_{[k-m,k)}(a))=k=\mu_{\mathrm{ad},m}(a)$.\\
Now we prove $\mu_{\mathrm{ad},m}(H)\geq\mathsf c_{\mathrm{ad},m}(H)$. We must prove that $\mathsf c_{\mathrm{ad},m}(a)\leq\mu_{\mathrm{ad},m}(H)$ for all $a\in H$. Assume to the contrary that there is some $a\in H$ such that $\mathsf c_{\mathrm{ad},m}(a)>\mu_{\mathrm{ad},m}(H)$. Let $k\in\N$ be minimal such that there is $a\in H$ with $\mathsf c_{\mathrm{ad},m}(a)=\mathsf d(\mathsf Z_k(a),\mathsf Z_{[k-m,k)}(a))$. If $\mathsf d(\mathsf Z_k(a),\mathsf Z_{[k-m,k)}(a))<k$, then there are some $x\in\mathsf Z_k(a)$ and $y\in\mathsf Z_{[k-m,k)}(a)$ such that $g=\gcd(x,y)\neq 1$. If $b=\pi_H(g^{-1}x)$, then $k-|g|,\,|y|-|g|\in\mathsf L(b)\cap [k-|g|-m,k-|g|]$ and
\[
 \mathsf c_{\mathrm{ad},m}(a)=\mathsf d(\mathsf Z_k(a),\mathsf Z_{[k-m,k)}(a))\leq\mathsf d(\mathsf Z_{k-|g|}(b),\mathsf Z_{[k-|g|-m,k-|g|)}(b))\leq\mathsf c_{\mathrm{ad},m}(b),
\]
and, by the minimal choice of $k$, we infer that $\mathsf c_{\mathrm{ad},m}(b)\leq\mu_{\mathrm{ad},m}(H)$, a contradiction.
\item If $H$ is not tame, then there is no $m\in\N$ with $m\geq\mathsf t(H)$. Thus we may assume that $\mathsf t(H)<\infty$.
Let $m\geq\mathsf t(H)$. By part~\ref{4.6.1}, it suffices to show that $\mu_{\mathrm{ad},m}(a)\leq\mathsf t(H)$ for all $a\in H$. Let $a\in H$. If $\mu_{\mathrm{ad},m}(a)=0$, then there is nothing to show. Now suppose $\mu_{\mathrm{ad},m}(a)=k>0$. Then we have $\mathsf L(a)\cap [k-m,k) = \lbrace l_1,\ldots, l_n\rbrace$ and $\mathsf d(\mathsf Z_k(a),\mathsf Z_{l_i}(a))=k$ for all $i\in[1,n]$. Then $\gcd(x,y)=1$ for all $x\in\mathsf Z_k(a)$ and $y\in\mathsf Z_{l_1}(a)$. Now let $x\in\mathsf Z_k(a)$, $y\in\mathsf Z_{l_1}(a)$, and choose $u\in\mathcal A(H)$ such that $y\in\mathsf Z(a)\cap u\mal H\mathsf Z(H)$. We find
\begin{multline}
\label{eq1}
\mathsf t(H)\geq\mathsf t(a,u\mal H)\geq
\mathsf d(x,\mathsf Z(a)\cap u\mal H\mathsf Z(H)) \\
=\min\lbrace\mathsf d(x,\mathsf Z_l(a)\cap u\mal H\mathsf Z(H))\mid l\in\mathsf L(a),\,l\neq k\rbrace
\geq\min\lbrace k,m+1\rbrace=k=\mu_{\mathrm{ad},m}(a),
\end{multline}
since $m+1>\mathsf t(H)$.
\qedhere
\end{enumerate}
\end{proof}

Another interesting observation arising from the proof of \autoref{4.6}.\ref{4.6.2} is the fact that the crucial inequality~\eqref{eq1} might fail for $m<\mathsf t(H)$ for some $a\in H$ (of course, with $\mu_{\mathrm{ad},m}(a)>0$). Additionally, \autoref{4.6}.\ref{4.6.2} can never be used to bound $\cadj(H)$, since $\cadj(H)=\mathsf c_{\mathrm{ad},m}(H)$ for $m=\min\triangle(H)$ if $|\triangle(H)|=1$, but then $\mathsf t(H)\geq m+2>m$, and therefore \autoref{4.6}.\ref{4.6.2} does not hold for $\cadj(H)$.\bigskip

\bigskip
\section{Applications to semigroup rings and generalized power series rings}
\label{sec:comp}
\bigskip

The arithmetic of such domains has attracted a lot of interest; for an overview, we refer to \cite{MR2265800} and \cite{MR1357822}; and for some recent results, we refer to \cite{MR2577137} and \cite{MR2483834}. Nevertheless, there are nearly no precise results on their arithmetic. In order to apply our monoid theoretic tools from \autoref{sec:cmon} and \cite{phil10} to the explicit computation of various arithmetical invariants of semigroup rings and generalized power series rings, we follow a $2$-step strategy. In the first step, we apply transfer principles as described in much detail in \cite[Section 3.2]{MR2194494}, and in the second step, we make use of the monoid theoretic tools.

\begin{definition}
\label{def:transfer}
A monoid homomorphism $\theta:H\rightarrow B$ is called a \emph{transfer homomorphism} if it has the following properties:
\begin{itemize}
\item[$\mathbf{T1}$] $B=\theta(H)\mal B$ and $\theta^{-1}(\mal B)=\mal H$.
\item[$\mathbf{T2}$] If $a\in H,\,r,\,s\in B$ and $\theta(a)=rs$, then there exist $b,\,c\in H$ such that $\theta(b)\sim r$, $\theta(c)\sim s$, and $a=bc$.
\end{itemize}
\end{definition}

\begin{definition}
Let $\theta:H\rightarrow B$ be a transfer homomorphism of atomic monoids and $\bar\theta:\mathsf Z(H)\rightarrow\mathsf Z(B)$ the unique homomorphism satisfying $\bar\theta(u\mal H)=\theta(u)\mal B$ for all $u\in\mathcal A(H)$. We call $\bar\theta$ the extension of $\theta$ to the factorization monoids.\\
For $a\in H$, the \emph{catenary degree in the fibers} $\mathsf c(a,\theta)$ denotes the smallest $N\in\N_0\cup\lbrace\infty\rbrace$ with the following property:
\begin{itemize}
 \item[] For any two factorizations $z,\,z'\in\mathsf Z(a)$ with $\bar\theta(z)=\bar\theta(z')$, there exists a finite sequence of factorizations $(z_0,z_1,\ldots,z_k)$ in $\mathsf Z(a)$ such that $z_0=z$, $z_k=z'$, $\bar\theta(z_i)=\bar\theta(z)$, and $\mathsf d(z_{i-1},z_i)\leq N$ for all $i\in[1,k]$; that is, $z$ and $z'$ can be concatenated by an $N$-chain in the fiber $\mathsf Z(a)\cap\bar\theta^{-1}((\bar\theta(z)))$.
\end{itemize}
Also, $\mathsf c(H,\theta)=\sup\lbrace\mathsf c(a,\theta)\mid a\in H\rbrace$ is called the \emph{catenary degree in the fibers} of $H$.
\end{definition}

We briefly fix the notation concerning sequences over finite abelian groups. Let $G$ be an additively written, finite abelian group. For a subset $A\subset G$ and an element $g\in G$, we set $-A=\lbrace -a\mid a\in A\rbrace$ and $A-g=\lbrace a-g\mid a\in A\rbrace$. Let $\mathcal F(G)$ be the free abelian monoid with basis $G$. The elements of $\mathcal F(G)$ are called \emph{sequences} over $G$. If a sequence $S\in\mathcal F(G)$ is written in the form $S=g_1\mdots g_l$, we tacitly assume that $l\in\N_0$ and $g_1,\ldots,g_l\in G$. For a sequence $S=g_1\mdots g_l$, we call
\begin{itemize}
 \item[] $|S|=l$ the \emph{length} of $S$,
 \item[] $\sigma(S)=\sum_{i=1}^lg_i\in G$ the \emph{sum} of $S$,
 \item[] $\supp(S)=\lbrace g_1,\ldots,g_l\rbrace\subset G$ the \emph{support} of $S$,
 \item[] $\Sigma(S)=\lbrace\sum_{i\in I}g_i\mid\emptyset\neq I\subset[1,l]\rbrace\subset G$ the \emph{set of subsums} of $S$, and
 \item[] $-\Sigma(S)=\lbrace\sum_{i\in I}(-g_i)\mid\emptyset\neq I\subset[1,l]\rbrace=\lbrace -g\mid g\in\Sigma(S)\rbrace\subset G$ the \emph{set of negative subsums} of $S$.
\end{itemize}
The sequence $S$ is called
\begin{itemize}
 \item a \emph{zero-sum sequence} if $\sigma(S)=0$,
 \item \emph{zero-sum free} if there is no non-trivial zero-sum subsequence, i.e.~$0\notin\Sigma(S)$, and
 \item a \emph{minimal zero-sum sequence} if $S$ is nontrivial, $\sigma(S)=0$, and every subsequence $S'\mid S$ with $1\leq |S'|<|S|$ is zero-sum free.
\end{itemize}
For a subset $G_0\subset G$, we set
\begin{align*}
\mathcal B(G_0) &= \lbrace S\in\mathcal F(G_0)\mid\sigma(S)=0\rbrace\mbox{ for the \emph{block monoid} over }G_0\mbox{ and} \\
\mathcal A(G_0) &= \lbrace S\in\mathcal F(G_0)\mid S\mbox{ minimal zero-sum sequence }\rbrace\subset\mathcal B(G_0).
\end{align*}
Then, in fact, $\mathcal B(G_0)$ is an atomic monoid and $\mathcal A(G_0)=\mathcal A(\mathcal B(G_0))$ is its set of atoms.\\
The \emph{Davenport constant} $\mathsf D(G_0)\in\N$ is defined to be the supremum of all lengths of sequences in $\mathcal A(G_0)$.

\begin{definition}
Let $G$ be an additive abelian group, $G_0\subset G$ a subset, $T$ a monoid, $\iota:T\rightarrow G$ a homomorphism, and $\sigma:\mathcal F(G_0)\rightarrow G$ the unique homomorphism such that $\sigma(g)=g$ for all $g\in G_0$. Then we call
\[
\mathcal B(G_0,T,\iota)=\lbrace St\in\mathcal F(G_0)\times T\mid\sigma(S)+\iota(t)=\mathbf 0\rbrace
\]
the \emph{$T$-block monoid} over $G_0$ defined by $\iota$.\\
If $T=\lbrace 1\rbrace$, then $\mathcal B(G_0,T,\iota)=\mathcal B(G_0)$ is the block monoid of all zero-sum sequences over $G_0$ and if $G_0=\lbrace\mathbf 0\rbrace$ then $\mathcal B(G_0,T,\iota)=[\mathbf 0]\times T$. Since $\mathbf 0\in\mathcal B(G_0,T,\iota)$ is prime, the arithmetic of $T$ and $\mathcal B(G_0,T,\iota)$ coincide in this situation.
\end{definition}

\begin{lemma}
\label{5.4}
Let $D$ be an atomic monoid, $P\subset D$ a set of prime elements, and $T\subset D$ an atomic submonoid such that $D=\mathcal F(P)\times T$. Let $H\subset D$ be a saturated atomic submonoid, let $G=\mathsf q(D/H)$ be its class group, let $\iota:T\rightarrow G$ be a homomorphism defined by $\iota(t)=[t]_{D/H}$, and suppose each class in $G$ contains some prime element from $P$.
\begin{enumerate}
\item\label{5.4.1} The map $\beta:H\rightarrow\mathcal B(G,T,\iota)$, given by $\beta(pt)=[p]_{D/H}+\iota(t)=[p]_{D/H}+[t]_{D/H}$, is a transfer homomorphism onto the $T$-block monoid over $G$ defined by $\iota$ and $\mathsf c(H,\beta)\leq 2$
\item\label{5.4.2} The following inequalities hold:
\begin{eqnarray*}
\mathsf c(\mathcal B(G,T,\iota))\leq &\mathsf c(H) &\leq\max\lbrace\mathsf c(\mathcal B(G,T,\iota)),\mathsf c(H,\beta)\rbrace,\\ 
\cmon(\mathcal B(G,T,\iota))\leq &\cmon(H) &\leq\max\lbrace\cmon(\mathcal B(G,T,\iota)),\mathsf c(H,\beta)\rbrace,\mbox{ and}\\
\mathsf t(\mathcal B(G,T,\iota))\leq &\mathsf t(H) &\leq\mathsf t(\mathcal B(G,T,\iota))+\mathsf D(G)+1.
\end{eqnarray*}
In particular, the equality $\mathsf c(H)=\mathsf c(\mathcal B(G,T,\iota))$ holds if $\mathsf c(\mathcal B(G,T,\iota))\geq 2$, and the equality $\cmon(H)=\cmon(\mathcal B(G,T,\iota))$ holds if $\cmon(\mathcal B(G,T,\iota))\geq 2$.
\item\label{5.4.3} $\mathcal L(H)=\mathcal L(\mathcal B(G,T,\iota))$, $\triangle(H)=\triangle(\mathcal B(G,T,\iota))$, $\min\triangle(H)=\min\triangle(\mathcal B(G,T,\iota))$, and $\rho(H)=\rho(\mathcal B(G,t,\iota))$.
\item\label{5.4.4} We set $\mathcal B=\lbrace S\in\mathcal B(G,T,\iota)\mid\mathbf 0\nmid S\rbrace$. Then $\mathcal B$ and $\mathcal B(G,T,\iota)$ have the same arithmetical properties, and
\begin{eqnarray*}
\mathsf c(\mathcal B)\leq &\mathsf c(H) &\leq\max\lbrace\mathsf c(\mathcal B),\mathsf c(H,\beta)\rbrace,\\ 
\cmon(\mathcal B)\leq &\cmon(H) &\leq\max\lbrace\cmon(\mathcal B),\mathsf c(H,\beta)\rbrace,\mbox{ and}\\
\mathsf t(\mathcal B)\leq &\mathsf t(H) &\leq\mathsf t(\mathcal B)+\mathsf D(G)+1.
\end{eqnarray*}
In particular, the equality $\mathsf c(H)=\mathsf c(\mathcal B)$ holds if $\mathsf c(\mathcal B)\geq 2$, and the equality $\cmon(H)=\cmon(\mathcal B)$ holds if $\cmon(\mathcal B)\geq 2$.\\
Additionally,
$\mathcal L(H)=\mathcal L(\mathcal B)$, $\triangle(H)=\triangle(\mathcal B)$, $\min\triangle(H)=\min\triangle(\mathcal B)$, and $\rho(H)=\rho(\mathcal B)$.
\end{enumerate}
\end{lemma}
\begin{proof}
\mbox{}
\begin{enumerate}
\item Follows by \cite[Proposition 3.2.3.3 and Proposition 3.4.8.2]{MR2194494}.
\item The assertion on the catenary degree follows by \cite[Theorem 3.2.5.5]{MR2194494}, the assertion on the monotone catenary degree by \cite[Lemma 3.2.6]{MR2194494}, and the assertion on the tame degree by \cite[Theorem 3.2.5.1]{MR2194494}.
\item Follows by \cite[Proposition 3.2.3.5]{MR2194494}.
\item Since $\mathbf 0\in\mathcal B(G,T,\iota)$ is a prime element, it defines a partition $\mathcal B(G,T,\iota)=[\mathbf 0]\times\mathcal B$ with $\mathcal B=\lbrace S\in\mathcal B(G,T,\iota)\mid\mathbf 0\nmid S\rbrace$. Thus all studied arithmetical invariants coincide for $\mathcal B$ and $\mathcal B(G,T,\iota)$. Now the assertions follow from part~\ref{5.4.2} and part~\ref{5.4.3}.
\qedhere
\end{enumerate}
\end{proof}

From now on, we write monoids additively. Then, for a reduced monoid $H$, $\mal H=\lbrace 0\rbrace$.

\begin{definition}
Let $K$ be a field and $H$ a reduced atomic monoid. Then we call
\begin{itemize}
 \item $K\ldbrack H\rdbrack=K\ldbrack X^s\mid s\in\mathcal A(H)\rdbrack$ the \emph{generalized power series ring}
 \item $K[H]=K[X^s\mid s\in\mathcal A(H)]$ the \emph{semigroup ring}
\end{itemize}
defined by $H$ over $K$.
\end{definition}

A submonoid $H\subset\N_0$ such that $\N_0\setminus H$ is finite is called a \emph{numerical monoid}. Then we set $\mathsf G(H)=\N_0\setminus H$ for the \emph{set of gaps} of $H$. Next we give a characterization of smooth numerical monoids.

While some abstract finiteness results on these rings can be obtained in a quite general setting, only very little is known about the explicit behavior of the various invariants of non-unique factorization theory.
For special generalized power series rings, we repeat the following finiteness result from \cite{MR2408326} in our terminology.
\begin{theorem}[cf. {\cite[Proposition 6.10 and Theorem 6.7]{MR2408326}}]
Let $K$ be a field and $H\subset\N_0$ a numerical monoid. Then $K\ldbrack H\rdbrack$ has finite catenary degree and finite elasticity.
\end{theorem}

In our further investigations, we will restrict ourselves generally to the most simple monoids possible, i.e., to numerical monoids, since even there the situation is very complicated once one begins to calculate everything in detail.

By \cite[Theorem 3.7.1]{MR2194494}, the arithmetic of weakly Krull domains, e.g., a generalized power series ring or a semigroup ring defined over a finite field, can mostly be described by studying appropriate $T$-block monoids, i.e., $\mathcal B(G,T,\iota)\subset\mathcal F(G)\times T$. Special cases of the generalized power series rings and the semigroup rings are studied in \cite[Example 3.7.3, Special Case 3.2 and 3.3]{MR2194494}. For special numerical monoids, i.e, $H=[e,\ldots,2e-1]$ with $e\geq 2$, and the generalized power series ring $R=F\ldbrack H\rdbrack$ over a finite field $F$, there a transfer homomorphism from $R$ onto $H$ is construced. In \autoref{specialtransfer}, we study for which monoids $H$, in general, this homomorphism is indeed a transfer homomorphism. Surprisingly, it turns out that one can characterize a special class of monoids, namely smooth monoids, this way.

\begin{definition}
Let $(H,\leq)$ be a reduced atomic monoid with a total order. Then we call $H$ \emph{smooth} if, for all $h,\,b,\,c\in H$ with $h\geq b+c$, we have $h-b\in H$.
\end{definition}

\begin{lemma}
\label{specialtransfer}
Let $K$ be a field, $(H,\leq)$ a reduced atomic monoid with total order, and set
\[
\phi:
\left\lbrace
\begin{array}{ccc}
\punkt{K\ldbrack H\rdbrack} & \rightarrow & H \\
f=\sum_{h\in H}f_hX^h & \mapsto & \min\lbrace h\in H\mid f_h\neq 0\rbrace\,.
\end{array}
\right.
\]
Then $\phi$ is a homomorphism, and $\phi$ is a transfer homomorphism if and only if $H$ is smooth.
\end{lemma}
\begin{proof}
Obviously, $\phi$ is a homomorphism. First we assume that $H$ is smooth. Now we must show the two axioms from \autoref{def:transfer}. $\mathbf{T1}$ is obvious, since $\phi$ is surjective. For $\mathbf{T2}$, let $u=\sum_{h\in H}u_hX^h\in K[H]$ and $b,\,c\in H$ with $\phi(u)=b+c$. We set $v=u_{\phi(u)}X^b$ and $w=X^c+\sum_{h\in H}u_hu_{\phi(u)}^{-1}X^{h-b}$. Then $v,\,w\in K\ldbrack H\rdbrack$, $vw=u$, and $\phi(w)=c$.\\
Second, we assume that $H$ is not smooth and show that axiom $\mathbf{T2}$ of \autoref{def:transfer} fails. Since $H$ is not smooth, there are $h,\,b,\,c\in H$ with $h>b+c$ $h-b,\,h-c\notin H$ and all three elements are minimal with this property. Then set $u=X^{b+c}+X^h$. Now $\phi(u)=b+c$. Let $h'\in\lbrace b,c\rbrace$ and set $v=X^{h'}$. Then $v\nmid u$ in $K\ldbrack H\rdbrack$, and thus $\mathbf{T2}$ does not hold.
\end{proof}

\begin{proposition}
\label{5.9}
Let $H\subset\N_0$ be a numerical monoid. Then the following are equivalent.
\begin{enumerate}
 \item \label{5.9.1} There is $m\in N$ such that $H=[m,m+1,\ldots,2m-1]$.
 \item \label{5.9.2} $\mathsf G(H)$ is an interval.
 \item \label{5.9.3} $H$ is smooth.
\end{enumerate}
\end{proposition}
\begin{proof}
Obviously, \ref{5.9.1} and \ref{5.9.2} are equivalent, and the implication from \ref{5.9.1} to \ref{5.9.3} is also clear.\\
\textbf{\ref{5.9.3} $\Rightarrow$ \ref{5.9.1}.} If $|\mathcal A(H)|=1$, then, trivially, $H=[1]=\N_0$. Now let $|\mathcal A(H)|>1$ and $k,\,n\in\N$ be minimal such that $[k,k+n]\cap\mathcal A(H)=\lbrace k,k+n\rbrace$. We proceed by distinguishing three cases.\\
\textbf{Case 1.} $n>k$. Then $2k<k+n$, and therefore $(k+n)-k=n\in H$. This is a contradiction to $k+n\in\mathcal A(H)$.\\
\textbf{Case 2.} $n=k$. Then $k+n=2k$ and this is again a contradiction to $k+n\in\mathcal A(H)$.\\
\textbf{Case 3.} $n<k$. By the minimality of $k$ and $n$, we have that $[k,k+n]\cap H=\lbrace k,k+n\rbrace$. Let $m\in\N$ be minimal such that $mn\geq k$. If $mn>k$, then $mn\in (k,kn)$, a contradiction; therefore $mn=k$, i.e., $n\mid k$. Let $x\in H\setminus [mn,(m+1)n,\ldots (2m-1)n]\cap [0,2mn)$. Additionally, $x>(m+a)n$ for some $a\in\N$, since $x\notin(mn,(m+1)n)$. Thus we find $x\in((m+1)n,2mn)\setminus n\N$. Then there exist $j\in[m+1,2m-1]$ and $l\in[1,n-1]$ such that $x=jn+l$. Let $j'\in[m-1,2m-1]$ be such that $j+j'=im$ for $i\geq 2$. Then we find $imn<imn+l=x+j'n\in H$. Thus we find $x+j'm-mn=(i-1)mn-l\in H$. This can now be repeated until we reach $mn+l\in H$---a contradiction. Thus $H=[mn,(m+1)n,\ldots,(2m+1)n]$ and, since $H$ is a numerical monoid, we have $\gcd(\mathcal A(H))=1$, and therefore $n=1$. Now the assertion follows.
\end{proof}

For the study of semigroup rings, the situation is even more difficult. Since there is then no transfer homomorphism $R=F[H]\rightarrow H$; see \cite[Example 3.7.3, Special Case 3.3]{MR2194494}. Thus---even after applying the transfer principles in order to be in an easier situation---it is necessary to compute all the invariants of non-unique factorization for more general $T$-block monoids, $\mathcal B(G,T,\iota)$, where neither $T$ nor $G$ are trivial.
In the upcoming subsections, we exploit the results from \cite{MR2243561}, \cite{MR2254337}, \cite[Proposition 16]{phil10} (repeated as \autoref{2.7}), \cite[Theorem 19.2]{phil10}, and \autoref{sec:cmon} (mainly \autoref{3.2}) together with recent programming techniques (see \cite{Ge-Li-Ph11} and \cite[Section 8]{Or-Ph-Sa-Sc10}) and parallelization to explicitly compute various arithmetical invariants, namely, the elasticity, the catenary degree, the monotone catenary degree, and a bound for the tame degree of the $T$-block monoids associated with the studied domains, and therefore for the domains themselves.

\smallskip
\subsection{Preliminaries about zero-sum sequences and \boldmath $T$-block monoids}
\label{sec:Tblock}
\mbox{}\\
In order to be able to describe the set of atoms of a $T$-block monoid precisely, we use the terminology of sequences over finite abelian groups.

For our algorithmic considerations in the forthcoming sections, it will be very useful to have some sort of order defined on the elements of  a finite abelian group $G$. By the structure theorem for finitely generated abelian groups, there are uniquely determined $r\in\N_0$ and $n_1,\ldots,n_r\in\N$ such that there is a group isomorphism $\varphi:G\rightarrow\Z/n_1\Z\times\ldots\times\Z/n_r\Z$ and $1<n_1\mid\ldots\mid n_r$. For $i\in[1,r)$, we choose $[0,n_i)$ as a system of representatives for $\Z/n_i\Z$. Now we can compare two elements $g_1,\,g_2\in G$ by comparing $\varphi(g_1)$ and $\varphi(g_2)$ with respect to the lexicographic order. For short, we simply write $g_1\leq g_2$ respectively $g_1\geq g_2$.

In particular, in \autoref{computing:section}, we will need some kind of coordinate representation for the elements of a $T$-block monoid, i.e., a monoid isomorphism mapping a $T$-block monoid onto a submonoid of $\Z^m\times\Z/n_1\Z\times\ldots\times\Z/n_r\Z$ for some $m,\,r\in\N_0$ and $n_1,\ldots,n_r\in\N$.
Let $G$ be a finite abelian group, $T$ a finitely generated monoid, and $\iota:T\rightarrow G$ a homomorphism. Let $T=D_1\times\ldots\times D_r$ be a product of finitely primary monoids $D_i\subset[p_1^{(i)},\ldots,p_{r_i}^{(i)}]\times\wmal{D_i}=\widehat{D_i}$ where $r_i\in\N$, and the $\wmal{D_i}$ are finitely generated abelian groups for $i\in[1,r]$. Then there are uniquely determined $l_i,\,k_i\in\N_0$ such that there is an isomorphism $\phi_i:\wmal{D_i}\rightarrow\Z^{l_i}\times\Z/n_1^{(i)}\Z\times\ldots\times\Z/n_{k_i}^{(i)}$ with $1<n_1^{(i)}\mid\ldots\mid n_{k_i}^{(i)}$ for $i\in[1,r]$. This isomorphism can be extended to an isomorphism $\bar\phi_i:\widehat{D_i}\rightarrow N_0^{r_i}\times\phi_i(\wmal{D_i})$ for $i\in[1,r]$. Now there is an isomorphism $\phi=\bar\phi_1\times\ldots\times\bar\phi_r:\widehat T\rightarrow\bar\phi_1(\widehat{D_1})\times\ldots\times\bar\phi_r(\widehat{D_r})$. This again can be extended to an isomorphism $\bar\varphi:\mathcal F(G)\times\widehat T\rightarrow N_0^{|G|}\times\phi(\widehat T)$. Now we can define the desired isomorphism by restriction of $\bar\varphi$ to the $T$-block monoid $\mathcal B(G,T,\iota)$ as follows:
\begin{equation}
\label{isormorphism}
\varphi=\bar\varphi|\mathcal B(G,T,\iota):\mathcal B(G,T,\iota)\rightarrow\bar\varphi(\mathcal B(G,T,\iota))\subset\N_0^{|G|}\times\prod_{i=1}^r\left(\N_0^{r_i}\times\Z^{l_i}\times\prod_{j=1}^{k_i}\Z/n_j^{(i)}\Z\right)
\,.
\end{equation}

\smallskip
\subsection{The set of atoms \boldmath $\mathcal A(G)$ of a block monoid}
\mbox{}\\
Based on ideas from \cite{Ge-Li-Ph11}, we give an algorithm for the computation of the set of atoms $\mathcal A(G)$ for a finite additive abelian group $G$. The problem of computing $\mathcal A(G)$ grows exponentially in terms of $|G|$, but, for very small groups as the ones involved in \autoref{example:section}, it can be easily performed---sometimes even by hand. Unfortunately, we have to do some sort of brute force search in the set of all $S\in\mathcal F(G)$ with $|S|\leq\mathsf D(G)$. But with the algorithm presented below, we can avoid most of the redundant checks and therefore speed up the computation dramatically.

\begin{algorithm}
\caption{Recursive Atom Search: $A\leftarrow\mathrm{RAS}(A,S,\Sigma,B)$}
\label{ras}
\begin{algorithmic}
\FORALL{$g\in B$}
\STATE $S'\leftarrow Sg$
\IF{$g\leq -\sigma(S')$}
\STATE $A\leftarrow A\cup\lbrace S'(-\sigma(S'))\rbrace$
\ENDIF
\STATE $\Sigma'\leftarrow\Sigma$
\STATE $B'\leftarrow\emptyset$
\FORALL{$g'\in B$}
\IF{$g+g'\in\Sigma$}
\STATE $\Sigma'\leftarrow\Sigma'\cup\lbrace g'\rbrace$
\ELSE
\STATE $B'\leftarrow B'\cup\lbrace g'\rbrace$
\ENDIF
\ENDFOR
\IF{$|B'|>0$}
\STATE $A\leftarrow\mathrm{RAS}(A,S',\Sigma',B')$
\ENDIF
\ENDFOR
\RETURN $A$
\end{algorithmic}
\end{algorithm}

\begin{algorithm}
\caption{Atoms Computation Algorithm 1: $\mathcal A(G) \leftarrow \mathrm{ACA1}(G)$}
\label{aca1}
\begin{algorithmic}
\STATE $A\leftarrow\lbrace\mathbf 0\rbrace$
\FORALL{$g\in G\setminus\lbrace\mathbf 0\rbrace$}
\IF{$g\leq -g$}
\STATE $A\leftarrow A\cup\lbrace g(-g)\rbrace$
\ENDIF
\STATE $\Sigma\leftarrow\lbrace\mathbf 0,g\rbrace$
\STATE $B\leftarrow G\setminus\lbrace\mathbf 0,-g\rbrace$
\STATE $S\leftarrow g$
\IF{$|B|>0$}
\STATE $A\leftarrow\mathrm{RAS}(A,S,\Sigma,B)$
\ENDIF
\ENDFOR
\RETURN $A$
\end{algorithmic}
\end{algorithm}

Since modular arithmetic on vectors with multiple coordinates is quite inefficient, it is necessary for a fast execution of the RAS, \autoref{ras}, to pre-compute the sums $g+g'$. This can be done once in the ACA1, \autoref{aca1}, before the main loop. For additional details on speeding up these types of algorithms by special alignment of the pre-computed data and on the parallelization aspects, the reader is referred to \cite[Section 3]{Ge-Li-Ph11}. 

\smallskip
\subsection{The set of atoms of a \boldmath $T$-block monoid}
\label{atomsofB:section}

\begin{lemma}
Let $G$ be a finite additive abelian group, $T$ a reduced atomic monoid, $\iota:T\rightarrow G$ a homomorphism, and $\mathcal B(G,T,\iota)\subset\mathcal F(G)\times T$ the $T$-block monoid over $G$ defined by $\iota$. Furthermore, suppose each class in $G$ contains some $p\in P$, and let $\bar\iota:\mathsf Z(T)\rightarrow\mathcal F(G)$ be the homomorphism generated by the extension of $\iota$ onto $\mathsf Z(T)$ such that, for a factorization $z=a_1\mdots a_n\in\mathsf Z(T)$ with $a_i\in\mathcal A(T)$ for $i\in[1,n]$, we have $\bar\iota(z)=\iota(a_1)\mdots\iota(a_n)$.\\
Then we have
\begin{multline}
\label{atomsofB}
 \mathcal A(\mathcal B(G,T,\iota))= \\
\lbrace
S\pi(z)\mid S\in\mathcal F(G),\,z\in\mathsf Z(T),\,S\bar\iota(z)\in\mathcal A(G),
\nexists n\geq 2:\,
\exists S_i\in\mathcal F(G),\,z_i\in\mathsf Z(T) \\
\mbox{with }S_i\bar\iota(z_i)\in\mathcal A(G)\mbox{ for }i\in[1,n]:\,
S_1\pi(z_1)\mdots S_n\pi(z_n)=S\pi(z)
\rbrace
\end{multline}
\end{lemma}
\begin{proof}
Clearly, every atom $a\in\mathcal A(\mathcal B(G,T,\iota))$ is of the form $a=S\pi(z)$ with $S\in\mathcal F(G)$, $z\in\mathsf Z(T)$, and $S\bar\iota(z)\in\mathcal A(G)$. Now suppose we have $n\in[2,\mathsf D(G)]$, $S_i\in\mathcal F(G)$, $z_i\in\mathsf Z(T)$, $S_i\bar\iota(z_i)\in\mathcal A(G)$ for $i\in[1,n]$ and $S\pi(z)=S_1\pi(z_1)\mdots S_n\pi(z_n)$. Obviously then, $a\notin\mathcal A(\mathcal B(G,T,\iota))$. Now the other inclusion is obvious.
\end{proof}
In general, it is very hard to calculate $\mathcal A(\mathcal B(G,T,\iota))$ explicitly by the characterization in \eqref{atomsofB}. But if we restrict ourselves to a finite group $G$ and a finitely generated reduced monoid $T$ such that $\mathcal A(G)$, $\mathcal A(T)$, and $\iota(a)$ for $a\in\mathcal A(T)$ are all known explicitly, we can formulate the ACA2, \autoref{aca2}, for the computation of the set of atoms of a $T$-block monoid.
\begin{algorithm}
\caption{Atoms Computation Algorithm 2: $\mathcal A(\mathcal B(G,T,\iota))\leftarrow\mathrm{ACA2}(G,T,\mathcal A(G),\mathcal A(T),\iota)$}
\label{aca2}
\begin{algorithmic}
\STATE $A\leftarrow\emptyset$
\STATE $D\leftarrow 0$
\FORALL{$S\in\mathcal A(G)$}
\IF{$|S|>D$}
\STATE $D\leftarrow |S|$
\ENDIF
\STATE $A\leftarrow A\cup\lbrace (S,1)\rbrace$
\ENDFOR
\STATE $F_0\leftarrow\emptyset$
\FORALL{$a\in\mathcal A(T)$}
\FORALL{$(S,1)\in A$}
\IF{$\iota(a)\mid S$}
\STATE $F_0\leftarrow F_0\cup\lbrace (\iota(a)^{-1}S,a)\rbrace$
\ENDIF
\ENDFOR
\ENDFOR
\STATE $E\leftarrow\emptyset$
\STATE $n\leftarrow 1$
\WHILE{$n<D$ and $F_{n-1}\neq\emptyset$}
\STATE $E\leftarrow E\cup F_{n-1}$
\STATE $E\leftarrow EF_0$
\STATE $F_n\leftarrow\emptyset$
\FORALL{$a\in\mathcal A(T)$}
\FORALL{$(S,b)\in A$}
\IF{$\iota(a)\mid S$}
\STATE $F_n\leftarrow F_n\cup\lbrace (\iota(a)^{-1}S,ab)$
\ENDIF 
\ENDFOR
\ENDFOR
\STATE $n\leftarrow n+1$
\ENDWHILE
\RETURN $A\cup F_0\cup\ldots\cup F_{n-1}$
\end{algorithmic}
\end{algorithm}

\smallskip
\subsection{Computing arithmetical invariants of a \boldmath $T$-block monoid}
\label{computing:section}
\mbox{}\\
Throughout this section, we implicitly use the isomorphism defined in~\eqref{isormorphism}. Thus we only have to work with submonoids
\[
S\subset\Z^m\times\Z/n_1\Z\times\ldots\times\Z/n_r\Z
\quad\mbox{with }m,\,r\in\N_0\mbox{ and }n_1,\ldots,n_r\in\N
\]
such that $S\cong\varphi(\mathcal B(G,T,\iota))$, where $G$ is an additively written finite abelian group, $T$ is a product of finitely many reduced finitely primary monoids of rank $1$, $\iota:G\rightarrow T$ is a homomorphism, and $\varphi$ is the isomorphism defined in~\eqref{isormorphism}. If $T$ is not the product of finitely many reduced finitely primary monoids of rank $1$, then $T$ would not be finitely generated.
Now we know $\mathcal A(S)$ explicitly, since, obviously, $\mathcal A(S)=\varphi(\mathcal A(\mathcal B(G,T,\iota)))$ and $\mathcal A(\mathcal B(G,T,\iota))$ can be computed explicitly by the ACA2; see \autoref{aca2}.

For the computation of the tame degree, we use the definition of the distance of factorizations and \cite[Theorem 19.2]{phil10}; for additional reference on this computation, see \cite[Section 4]{MR2243561}.

Now we are ready to describe the computation step by step.

\smallskip
\subsubsection{Finding the elements of $\mathcal A(\sim_S)$}
\label{step:1}
\mbox{}\\
The first step is finding the elements of $\mathcal A(\sim_S)$ explicitly. Unfortunately, this is a very hard task. Probably, the most efficient way is the following one as described in \cite[Sections 1 and 2]{MR2254337}.
\begin{enumerate}
 \item Since we know $\mathcal A(S)$ explicitly, we can write the atoms of $S$ in their coordinates as vectors:
\[
\mathcal A(S)=\lbrace (a_1^{(1)},\ldots,a_m^{(1)},a_{m+1}^{(1)}\mod n_1,\ldots,a_{m+r}^{(1)}\mod n_r),\ldots\rbrace
\,.
\]
 \item By \cite[Section 2]{MR2254337}, finding the elements of $\mathcal A(\sim_S)$ is equivalent to determining the minimal positive solutions of the following system of linear diophantine equations:
\begin{equation}
\label{eq:system}
\begin{array}{lccclclcccll}
 x_1a_1^{(1)}        & + & \ldots & + & x_ka_1^{(k)}        & - & y_1a_1^{(1)}     & - & \ldots & - & y_ka_1^{(k)}     & = 0 \\
 \vdots                   &    &          &    & \vdots                  &    & \vdots                &   &          &   & \vdots                & \vdots \\
 x_1a_m^{(1)}       & + & \ldots & + & x_ka_m^{(k)}       & - & y_1a_m^{(1)}     & - & \ldots & - & y_ka_m^{(k)}     & = 0 \\
 x_1a_{m+1}^{(1)} & + & \ldots & + & x_ka_{m+1}^{(k)} & - & y_1a_{m+1}^{(1)} & - & \ldots & - & y_ka_{m+1}^{(k)} & \equiv 0\mod n_1 \\
 \vdots                   &    &          &    & \vdots                   &    & \vdots                &   &          &   & \vdots                & \vdots \\
 x_1a_{m+r}^{(1)} & + & \ldots & + & x_ka_{m+r}^{(k)} & - & y_1a_{m+r}^{(1)} & - & \ldots & - & y_ka_{m+r}^{(k)} & \equiv 0\mod n_r
\end{array}
\end{equation}
We write a solution $(x_1,\ldots,x_k,y_1,\ldots,y_k)$ as $((x_1,\ldots,x_k),(y_1,\ldots,y_k))$.
 \item Again, by \cite[Section 2]{MR2254337} and \cite[Section 2]{MR1484089}, finding the set of minimal positive solutions is equivalent to finding the set of minimal positive solutions for the following enlarged system and then projecting back by the map and removing the zero element (if appearing after the projection) from the set of solutions:
\begin{equation}
\label{system}
\begin{array}{lccclccclcll}
 x_1a_1^{(1)}     & + & \ldots & - & y_1a_1^{(1)}     & - & \ldots &   &            &   & & = 0 \\
 \vdots                &    &          &   & \vdots                &    &          &  &            &   & & \vdots \\
 x_1a_m^{(1)}     & + & \ldots & - & y_1a_m^{(1)}     & - & \ldots &   &            &   & & = 0 \\
 x_1a_{m+1}^{(1)} & + & \ldots & - & y_1a_{m+1}^{(1)} & - & \ldots & + & x_{k+1}n_1 & - & y_{k+1}n_1 & = 0 \\
 \vdots & & & & \vdots & & & & \vdots & & \vdots & \vdots \\
 x_1a_{m+r}^{(1)} & + & \ldots & - & y_1a_{m+r}^{(1)} & - & \ldots & + & x_{k+r}n_r & - & y_{k+r}n_r & = 0
\end{array}
\end{equation}
\[
 \Phi:\left\lbrace
\begin{array}{ccc}
\N_0^{k+r}\times\N_0^{k+r} &\rightarrow &\N_0^k\times\N_0^k \\
((x_1,\ldots,x_{k+r}),(y_1,\ldots,y_{k+r})) &\mapsto &((x_1,\ldots,x_k),(y_1,\ldots,y_k)).
\end{array}
\right.
\]
One of the most efficient algorithms for finding these solutions is due to Contejean and Devie; see~\cite{MR1283022}. Nevertheless, this might take a very long time since the problem of determining the set of all minimal non-negative solutions of a system of linear diophantine equations is well known to be NP-complete.
\end{enumerate}

\smallskip
\subsubsection{Removing unnecessary elements}
\label{step:2}
\mbox{}\\
Clearly, elements of the form $((0,\ldots,0,1,0,\ldots,0),$ $(0,\ldots,0,1,0,\ldots,0))$ are minimal solutions. But as elements of $\mathcal A(\sim_S)$, these elements do not carry any information about the arithmetic of $S$. Therefore we may simply drop them. Since, for any two factorizations, $(x,y)\in\mathsf Z(S)$ is equivalent to $(y,x)\in\mathsf Z(S)$, we may also reduce the number of pairs by a factor of two. This smaller set will be denoted by $\mathcal A(\sim_S)^*=\lbrace ((x_1,\ldots,x_k),(y_1,\ldots,y_k)),\ldots\rbrace$.

\smallskip
\subsubsection{Computing the elasticity}
\label{step:3}
\mbox{}\\
By our finiteness assumptions on $T$, i.e., since $T$ is finitely generated, we know this set is finite. Thus we can simply compute the elasticity using \cite[Proposition 14.2]{phil10} as follows:
\[
\rho(S)=\max\left\lbrace\left.\frac{x_1+\ldots+x_k}{y_1+\ldots+y_k},\frac{y_1+\ldots+y_k}{x_1+\ldots+x_k}\right| ((x_1,\ldots,x_k),(y_1,\ldots,y_k))\in\mathcal A(\sim_S)^*\right\rbrace
\,.
\]

\smallskip
\subsubsection{Computing the catenary degree}
\label{step:4}
\mbox{}\\
By \autoref{2.7}.\ref{2.7.2}, we need only consider elements $a\in S$ such that their factorizations appear as part of an element of $\mathcal A(\sim_S)$ and such that their sets of factorizations consist of more than one $\mathcal R$-equivalence class. Then we get the catenary degree by taking the maximum over $\mu(a)$ for all those $a\in S$.

\begin{algorithm}
\caption{Recursive $\mathcal R$-Class Finder: $\mathcal R\leftarrow\mathrm{RCF}(\mathcal R,\mathsf Z=\lbrace z_1,\ldots,z_n\rbrace)$}
\label{rcf}
\begin{algorithmic}
\STATE $r\leftarrow\lbrace z_1\rbrace$
\STATE $\mathsf Z\leftarrow\mathsf Z\setminus\lbrace z_1\rbrace$
\STATE $n\leftarrow n-1$
\STATE $\mathsf Z=\lbrace z_1,\ldots,z_n\rbrace$ \COMMENT{renumber}
\STATE $i\leftarrow 1$
\WHILE{$i<n$}
\FOR{$i=1$ to $n$}
\IF{$\gcd(z_i,x)\neq 1$ for some $x\in r$}
\STATE $r\leftarrow r\cup\lbrace z_i\rbrace$
\STATE $\mathsf Z\leftarrow\mathsf Z\setminus\lbrace z_i\rbrace$
\STATE $n\leftarrow n-1$
\STATE $\mathsf Z=\lbrace z_1,\ldots,z_n\rbrace$ \COMMENT{renumber}
\STATE \bfseries break
\ENDIF
\ENDFOR
\ENDWHILE
\STATE $\mathcal R\cup\lbrace r\rbrace$
\IF{$\mathsf Z\neq\emptyset$}
\STATE $\mathcal R\leftarrow\mathrm{RCF}(\mathcal R,\mathsf Z)$
\ENDIF
\RETURN $\mathcal R$
\end{algorithmic}
\end{algorithm}

\begin{algorithm}
\caption{Catenary degree Computation Algorithm: $\mathsf c(S)\leftarrow\mathrm{CCA}(\mathcal A(S),\mathcal A(\sim_S)^*)$}
\label{cca}
\begin{algorithmic}
\STATE $A\leftarrow\emptyset$
\FORALL{$(x,y)\in\mathcal A(\sim_S)^*$}
\STATE $A\leftarrow A\cup\lbrace\pi(x)\rbrace$
\ENDFOR
\STATE $c\leftarrow 0$
\FORALL{$a\in A$}
\STATE $\mathcal R_a\leftarrow\mathrm{RCF}(\mathsf Z(a))$
\IF{$|\mathcal R_a|>1$}
\STATE $\mu\leftarrow\min\lbrace |x|\mid\mathcal R_a\rbrace$
\IF{$c<\mu$}
\STATE $c\leftarrow\mu$
\ENDIF
\ENDIF
\ENDFOR
\RETURN $c$
\end{algorithmic}
\end{algorithm}

\smallskip
\subsubsection{Computing the tame degree}
\label{step:5}
\mbox{}\\
After having computed $\mathsf Z(a)$ for all $a\in S$ such that $\mathcal A_a(\sim_S)\neq\emptyset$, we can apply \cite[Theorem 19.1]{phil10} for every $u\in\mathcal A(S)$. Since there are only finitely many, we get the tame degree as the maximum of these values.

\smallskip
\subsubsection{Computing the monotone catenary degree}
\label{step:6}
\mbox{}\\
For computing the monotone catenary degree, we compute the equal catenary degree $\ceq(S)$ and the adjacent catenary degree $\cadj(S)$.
We start with the adjacent catenary degree and proceed like in~\ref{step:1}. We use the fact that $\sim_{S,\mathrm{mon}}=\lbrace (x,y)\in\sim_S\mid |x|\leq |y|\rbrace$ and again \cite[Section 2]{MR2254337}. Now finding the elements of $\mathcal A(\sim_{S,\mathrm{mon}})$ is equivalent to determining the minimal positive solutions of a system of linear diophantine equations.

Before we construct this finite system of linear diophantine equations explicitly, we formulate a short lemma.

\begin{lemma}
\label{5.11}
Let $H$ be a finitely generated monoid.\\
Then $\sim_{H,\mathrm{mon}}$ is a finitely generated Krull monoid.
\end{lemma}
\begin{proof}
Let $H$ be a finitely generated monoid. Since $\sim_H\subset\mathsf Z(H)\times\mathsf Z(H)$ is then a saturated submonoid of a finitely generated monoid, $\sim_H$ is finitely generated by \cite[Proposition 2.7.5]{MR2194494}. Now assume $\sim_H$ has $n\in\N$ generators. Then the atoms of $\sim_H$ can be described as the minimal solutions of a system of finitely many, say $k$, linear diophantine equations in $2n$ variables as in step \ref{step:1} above. Then the atoms of $\sim_{H,\mathrm{mon}}$ can be described as the minimal solutions of a system of $k+1$ linear diophantine equations in $2n+1$ variables---see below for the explicit description of this system of linear diophantine equations. Thus $H$ is a finitely generated Krull monoid by \cite[Theorem 2.7.14]{MR2194494}.
\end{proof}

The system is \eqref{eq:system}, with one additional variable $z$ and one equation, namely,
\[
x_1+\ldots+x_k-y_1-\ldots-y_k+z= 0.
\]
The coefficients at $z$ are zero in all other equations.
Now we have two possibilities.
\begin{itemize}
\item Either we proceed by the same steps as in \ref{step:1} and solve this directly
\item or we use the incremental version of the algorithm of Devie and Contejoud (see~\cite[Section 9]{MR1283022}) and the set $\mathcal A(\sim_S)$, which we already computed in \ref{step:1}.
\end{itemize}
Next we can reduce the set of relations which we must consider, as in \ref{step:2}.
By \autoref{3.2}.\ref{3.2.4}, we have to consider only elements $a\in S$ such that $\mathcal A_a(\sim_{S,\mathrm{mon}})\neq\emptyset$. Then we get the adjacent catenary degree by taking the maximum over $\muadj (a)$ for all those $a$.
For the computation of the equal catenary degree, we must know the elements of $\mathcal A(\sim_{S,\mathrm{eq}})$. But these are already known, since $\mathcal A(\sim_{S,\mathrm{eq}})\subset\mathcal A(\sim_{S,\mathrm{mon}})$. Here we can again reduce the set of relations which we must consider, as in \ref{step:2}. By \autoref{3.2}.\ref{3.2.4}, we have to consider only elements $a\in S$ such that $\mathcal A_a(\sim_{S,\mathrm{eq}})\neq\emptyset$ and $|\mathcal R_{a,k}|>1$ for some $k\in\mathsf L(a)$. Now this can be done by applying the RCF, \autoref{rcf}, to $\mathsf Z_k(a)$ instead of $\mathsf Z(a)$. Then we get the equal catenary degree by taking the maximum over $\mueq(a)$ for all those $a$.

Now we find the monotone catenary degree by $\cmon(S)=\max\lbrace\cadj(S),\ceq(S)\rbrace$.

\smallskip
\subsubsection{Reducing the computation time for the catenary degree}
\label{step:7}
\mbox{}\\
If we are only interested in the computation of the catenary degree, we can speed up the very time consuming computations in \autoref{step:1} in the following way. In favor of \autoref{2.7}.\ref{2.7.2}, we may restrict our search for minimal solutions of the system of linear diophantine equations~\eqref{system} to solutions $(x_1,\ldots,x_{k+r},y_1,\ldots,y_{k+r})$ such that $\sum_{i=1}^kx_i\leq\mathsf c(S)$ and $\sum_{i=1}^ky_i\leq\mathsf c(S)$. Of course, we do not know $\mathsf c(S)$ a priori, but we may replace it with any upper bound---the better the bound, the faster the computation. In our special situation of $T$-block monoids, we can find a reasonably good bound by \cite[Theorem 3.6.4.1]{MR2194494} and by \cite[Proposition 3.6.6]{MR2194494}. Formulated in our terminology, these results read as follows.
\begin{theorem}
\label{5.12}
Let $G$ be an additively written abelian group, $T$ a reduced finitely generated monoid, $\iota:T\rightarrow G$ a homomorphism, and $\mathcal B(G,T,\iota)\subset\mathcal F(G)\times T$ the $T$-block monoid over $G$ defined by $\iota$. Then
\begin{enumerate}
\item $\rho(\mathcal B(G,T,\iota),\mathcal F(G)\times T)\leq\rho(T)$.
\item $\mathsf c(\mathcal B(G,T,\iota))\leq\rho(T)\mathsf D(G)\max\lbrace\mathsf c(T),\mathsf D(G)\rbrace$.
\end{enumerate}
\end{theorem}
Now we set $C=\rho(T)\mathsf D(G)\max\lbrace\mathsf c(T),\mathsf D(G)\rbrace$ for the upper bound. Though this does not speed up the search for minimal solutions itself that much, it is a very efficient (additional) termination criterion in our variant of the algorithm due to Contejean and Devie; for reference on the originally proposed algorithm, see~\cite{MR1283022}.\\
Unfortunately, this method has one drawback for the computation of the elasticity and the tame degree. As we no longer compute all minimal solutions to our system of linear diophantine equations, we no longer compute all elements in $\mathcal A(\sim_S)$, and therefore we cannot compute more than a lower bound for the elasticity in \autoref{step:3} and for the tame degree in \autoref{step:5}.

\smallskip
\subsubsection{Computing the elasticity from an appropriate subset of $\mathcal A(\sim_S)$}
\label{step:8}
\mbox{}\\
In \cite{MR1135355}, Domenjoud proposed an algorithm for computing the set of minimal solutions of a system of linear diophantine equations, which computes the set of minimal solutions with minimal support in a first step. All other minimal solutions can then be found by ``appropriate'' linear combinations of them using non-negative rational coefficients. With this interesting fact in mind, we consider the following lemma.

\begin{definition}
Let $H$ be an atomic monoid. For $x\in\mathsf Z(H)$, we set
\[
\supp(x)=\lbrace u\in\mathcal A(H_{\mathrm{red}})\mid u\mid x\rbrace.
\]
\end{definition}

\begin{lemma}
Let $H$ be a finitely generated monoid. Then
\[
 \rho(H)=\sup\left\lbrace\left.\frac{|x|}{|y|}\right|(x,y)\in\mathcal A'(\sim_H)\right\rbrace\,,
\]
where $\mathcal A'(\sim_H)=\lbrace(x,y)\in\mathcal A(\sim_H)\mid\supp(x)\cup\supp(y)\mbox{ is minimal}\rbrace$.
\end{lemma}

\begin{proof}
Let $(x,y)\in\mathcal A(\sim_H)$. Then there are $n\in\N$, $(x_i,y_i)\in\mathcal A'(\sim_H)$, and $q_i\in\Q$ with $0\leq q_i<1$ for $i\in[1,n]$ such that
\[
 (x,y)=\prod_{i=1}^n(x_i,y_i)^{q_i}.
\]
Such a decomposition exists, since the equivalent one exists for the set of solutions of the associated system of linear diophantine equations, see \cite[Theorem 3]{MR1135355}.
When we pass to the lengths, we find $|x|=\sum_{i=1}^n q_i|x_i|$ and $|y|=\sum_{i=1}^n q_i|y_i|$. This yields
\[
\frac{|x|}{|y|}\cdot |y|=|x|=\sum_{i=1}^nq_i|x_i|=\sum_{i=1}^nq_i\frac{|x_i|}{|y_i|}|y_i|\leq\max_{i=1}^n\frac{|x_i|}{|y_i|}\sum_{i=1}^nq_i|y_i|=\max_{i=1}^n\frac{|x_i|}{|y_i|}\cdot|y|\,.
\]
Thus we find
\[
 \frac{|x|}{|y|}\leq\max_{i=1}^n\frac{|x_i|}{|y_i|}.
\]
Since $\mathcal A'(\sim_H)\subset\mathcal A(\sim_H)$, the assertion now follows by \cite[Proposition 14.2]{phil10}.
\end{proof}

Thus we can restrict ourselves to minimal solutions with minimal support for computing the elasticity.\\
As far as computational performance is concerned, the most interesting point of this approach is that there are straightforward optimizations of Domenjoud's algorithm for symmetric systems of linear diophantine equations like the one in~\eqref{system}.

\smallskip
\subsection{Three explicit examples}
\label{example:section}
\mbox{}\\
Let $p\in\mathbb P$ be a prime. Let $R=\F_p[X^{e_1},X^{e_2}]$ with $e_1,\,e_2\geq 2$. Then $R$ is a one-dimensional noetherian domain with integral closure $\widehat R=\F_p[X]$ and conductor $\mathfrak f=(R:\widehat R)=X^f\widehat R$, where $X\in\widehat R$ is a prime element and $f$ is the frobenius number of the numerical monoid generated by $e_1$ and $e_2$. Thus $R$ is an order in the Dedekind domain $\widehat R$, and $X\widehat R$ is the only maximal ideal of $\widehat R$ containing $\mathfrak f$. Furthermore, $\wmal R=\mal R=\mal\F_p$. By the computations in \cite[Special case 3.2 in Example 3.7.3]{MR2194494}, we have $G=\pic(R)\cong\F_p$.

For a more detailed presentation of the explicit computations, the reader is referred to \cite[Section 2.3.5]{phd}.

\smallskip
\subsubsection{$\F_3[X^2,X^3]$}
\mbox{}\\
Now let $p=3$. Since $|G|=3$, we write $G=\lbrace\mathbf 0,\mathbf e,\mathbf{e'}\rbrace$. Clearly---or by applying the ACA1, see \autoref{aca1}---we have $\mathcal A(G)=\lbrace\mathbf 0,\mathbf{ee'},\mathbf e^3,\mathbf{e'}^3\rbrace$. Now we apply \cite[Theorem 3.7.1]{MR2194494} and switch to the block monoid, which is a $T$-block monoid over $G$, say $\mathcal B(G,T,\iota)\subset\mathcal F(G)\times T$, where $T$ is the reduced finitely primary monoid generated by $\mathcal A(T)=\lbrace X^ng\mid n\in\lbrace 2,3\rbrace,\,g\in G\rbrace$ and $\iota$ is the uniquely determinated homomorphism $\iota:G\rightarrow T$ such that $\iota(X^ng)=g$ for all $n\in\lbrace 2,3\rbrace$ and $g\in G$.\\
Now we apply the ACA2, see \autoref{aca2}, and find
\begin{align*}
\mathcal A(\mathcal B(G,T,\iota))=\lbrace &(\mathbf 0,1),(\mathbf{ee'},1),(\mathbf e^3,1),(\mathbf{e'}^3,1),(1,X^2\mathbf 0),(1,X^3\mathbf 0),(\mathbf e,X^2\mathbf{e'}),(\mathbf e,X^3\mathbf{e'}),\\
&(\mathbf{e'},X^2\mathbf e),(\mathbf{e'},X^3\mathbf e),(\mathbf e^2,X^2\mathbf e),(\mathbf e^2,X^3\mathbf e),(\mathbf{e'}^2,X^2\mathbf{e'}),(\mathbf{e'}^2,X^3\mathbf{e'})\rbrace.
\end{align*}
Using the construction from the beginning of \autoref{computing:section}, we find
\[
\widehat T\cong\N_0\times\Z/3\Z
\quad\mbox{and}\quad
\mathcal B(G,T,\iota)\cong S\subset\N_0^4\times\Z/3\Z
\,.
\]
Then, for the set of atoms, we find
\begin{align*}
\mathcal A(S)=\lbrace 
&(1,0,0,0,\bar 0),(0,1,1,0,\bar 0),(0,3,0,0,\bar 0),(0,0,3,0,\bar 0),(0,0,0,2,\bar 0),\\
&(0,0,0,3,\bar 0),(0,1,0,2,\bar 2),(0,1,0,3,\bar 2),(0,0,1,2,\bar 1),(0,0,1,3,\bar 1),\\
&(0,2,0,2,\bar 1),(0,2,0,3,\bar 1),(0,0,2,2,\bar 2),(0,0,2,3,\bar 2)\rbrace\,.
\end{align*}
Since the atom $(1,0,0,0,\bar 0)$ is prime, we can restrict to a monoid $\bar S\subset\N_0^3\times Z/3\Z$.
Now we can find everything by using the algorithms presented at the end of \autoref{computing:section}.
Even in the modified version of the algorithm in \autoref{step:1}---here the bound is $13.5$---we find about 7,500 minimal representations to consider after the reduction in \autoref{step:2}.

From those, we get $\mathsf c(\F_3[X^2,X^3])=3$ in \autoref{step:4}. Since we did not compute all minimal solutions, we find $\mathsf t(\F_3[X^2,X^3])\geq 4$ in \autoref{step:5}.\\
By using the alternative approach from \autoref{step:8}, we find $\rho(\F_3[X^2,X^3])=\frac{5}{2}$. Note that this particular result on the elasticity can also be obtained by \cite[Example 3.7.3, Special Case 3.2]{MR2194494}.

\smallskip
\subsubsection{$\F_2[X^2,X^3]$}
\mbox{}\\
Let $p=2$. Then $|G|=2$, so write $G=\lbrace\mathbf 0,\mathbf e\rbrace$. Obviously---or by applying the ACA1, see \autoref{aca1}---we have $\mathcal A(G)=\lbrace\mathbf 0,\mathbf e^2\rbrace$. Now we apply \cite[Theorem 3.7.1]{MR2194494} as in the case $p=3$ and switch to the block monoid, which is a $T$-block monoid over $G$, say $\mathcal B(G,T,\iota)\subset\mathcal F(G)\times T$, where $T$ is the reduced finitely primary monoid generated by $\mathcal A(T)=\lbrace X^ng\mid n\in\lbrace 2,3\rbrace,\,g\in G\rbrace$ and $\iota$ is the uniquely determinated homomorphism $\iota:G\rightarrow T$ such that $\iota(X^ng)=g$ for all $n\in\lbrace 2,3\rbrace$ and $g\in G$.\\
Now we apply the ACA2, see \autoref{aca2}, as before and find
\[
\mathcal A(\mathcal B(G,T,\iota))=\lbrace (\mathbf 0,1),(\mathbf e^2,1),(1,X^2\mathbf 0),(1,X^3\mathbf 0),(\mathbf e,X^2\mathbf e),(\mathbf e,X^3\mathbf e)\rbrace
\,.
\]

Using the construction from the beginning of \autoref{computing:section}, we find
\[
\widehat T\cong\N_0\times\Z/2\Z
\quad\mbox{and}\quad
\mathcal B(G,T,\iota)\cong S\subset\N_0^3\times\Z/2\Z
\,.
\]
Then, for the set of atoms, we find 
\[
\mathcal A(S) = \lbrace
(1,0,0,\bar 0),(0,2,0,\bar 0),(0,0,2,\bar 0),(0,0,3,\bar 0),(0,1,2,\bar 1),(0,1,3,\bar 1)
\rbrace\,.
\]
Since the atom $(1,0,0,\bar 0)$ is prime, we can use the same arguments as in \autoref{5.4}.\ref{5.4.4} and restrict to a monoid $\bar S\subset\N_0^2\times\Z/2\Z$ with a reduced set of atoms.

Since, in this case, \autoref{step:1} can be performed easily without any bound, we compute all atoms. 
Given this list, we immediately find $\rho(\F_2[X^2,X^3])=2$ in \autoref{step:3}. Note that this particular result on the elasticity can also be obtained by \cite[Example 3.7.3, Special Case 3.2]{MR2194494}.\\
Now we proceed with \autoref{step:4} and we deduce $\mathsf c(\F_2[X^2,X^3])=3$ and $\mathsf D(\F_2)+1+\mathsf t(S)=6\geq\mathsf t(\F_2[X^2,X^3])\geq\mathsf t(S)=3$.

Next, we compute the monotone catenary degree. For this, we proceed as in \autoref{step:6} and find $\cadj(S)=3$ and $\ceq(S)=3$, and thus $\cmon(\F_2[X^2,X^3])=\cmon(S)=3$.

\smallskip
\subsubsection{$\F_2[X^2,X^5]$}
\mbox{}\\
The results in this case differ slightly from then ones we obtained above. We have $|G|=2$, say $G=\lbrace\mathbf 0,\mathbf e\rbrace$. Again, we have $\mathcal A(G)=\lbrace\mathbf 0,\mathbf e^2\rbrace$. Now we apply \cite[Theorem 3.7.1]{MR2194494} as before and switch to the block monoid, which is a $T$-block monoid over $G$, say $\mathcal B(G,T,\iota)\subset\mathcal F(G)\times T$, where $T$ is the reduced finitely primary monoid generated by $\mathcal A(T)=\lbrace X^ng\mid n\in\lbrace 2,5\rbrace,\,g\in G\rbrace$ and $\iota$ is the uniquely determinated homomorphism $\iota:G\rightarrow T$ such that $\iota(X^ng)=g$ for all $n\in\lbrace 2,5\rbrace$ and $g\in G$.\\
Now we apply the ACA2, see \autoref{aca2}, as before and find
\[
\mathcal A(\mathcal B(G,T,\iota))=\lbrace (\mathbf 0,1),(\mathbf e^2,1),(1,X^2\mathbf 0),(1,X^5\mathbf 0),(\mathbf e,X^2\mathbf e),(\mathbf e,X^5\mathbf e)\rbrace
\,.
\]
Using the construction from the beginning of \autoref{computing:section}, we find
\[
\widehat T\cong\N_0\times\Z/2\Z
\quad\mbox{and}\quad
\mathcal B(G,T,\iota)\cong S\subset\N_0^3\times\Z/2\Z
\,.
\]
Then, for the set of atoms, we find
\[
\mathcal A(S) = \lbrace
(1,0,0,\bar 0),(0,2,0,\bar 0),(0,0,2,\bar 1),(0,0,5,\bar 1),(0,1,2,\bar 1),(0,1,5,\bar 1)
\rbrace\,.
\]
Since the atom $(1,0,0,\bar 0)$ is prime, we can use the same arguments as in \autoref{5.4}.\ref{5.4.4} and restrict to a monoid $\bar S\subset\N_0^2\times\Z/2\Z$ with a reduced set of atoms.

Since, in this case, \autoref{step:1} can be performed without any bound, we compute all atoms. Now, we find a list of $25$ atoms after \autoref{step:2}. Given this list, we immediately find $\rho(\F_2[X^2,X^5])=3$ in \autoref{step:3}.
Now we proceed with \autoref{step:4} and obtain $\mathsf t(S)=4$, $\mathsf c(\F_2[X^2,X^5])=5$, and $\mathsf D(\F_2)+1+\mathsf t(S)=7\geq\mathsf t(\F_2[X^2,X^5])\geq\max\lbrace\mathsf t(S),\mathsf c(\F_2[X^2,X^5])\rbrace=5$.
Next, we compute the monotone catenary degree. For this, we proceed as in \autoref{step:6} and start with the adjacent catenary degree. We find $\cadj(S)=5$. Next we compute the equal catenary degree and find $\ceq(S)=6$. Now we find $\cmon(\F_2[X^2,X^5])=\cmon(S)=6>5=\mathsf c(\F_2[X^2,X^5])$.

\end{document}